\documentclass{amsart}

\usepackage{amsmath,amsfonts, mathtools, amsthm, relsize, commath, bbm,
	mathrsfs}

\usepackage[svgnames]{xcolor}

\usepackage{url}
\makeatletter
\renewcommand{\@setdate}{%
  \begin{flushleft}
    \footnotesize \@date
  \end{flushleft}
}\makeatother
\newtheorem{theorem}{Theorem}[subsection]
\newtheorem{lemma}{Lemma}[subsection]
\newtheorem{prop}{Proposition}[subsection]
\newtheorem{corollary}{Corollary}[subsection]
\newtheorem{definition}{Definition}[subsection]
\newtheorem{example}{Example}[subsection]

\newcommand{\unitcircle}{\mathbb{S}^1}
\newcommand{\image}[2]{#1\left(#2\right)}
\newcommand{\invimage}[2]{\image{#1^{-1}}{#2}}
\newcommand{\inv}[1]{#1^{-1}}
\newcommand{\Chi}{\mathlarger{\chi}}
\newcommand{\seq}[3]{\left(#1_#2\right)_{#2 \in #3}}
\newcommand{\dual}[1]{\widehat{#1}}
\newcommand{\orth}[1]{#1^{\bot}}
\newcommand{\closure}[1]{\overline{#1}}

\usepackage{hyperref}
\hypersetup{colorlinks=true,urlcolor=blue,linkcolor=red,citecolor=green}

\title[Locally compact modules over localizations of the integers]{Structure Theorems for locally compact modules over localizations of the integers}
\author{Pedro Lourenço}
\date{February 2026}
\address{Department of Mathematics, University of Porto, Portugal}
\email{pemane2002@gmail.com}

\begin{document}

\thispagestyle{empty}
\addtocounter{page}{-1}

\begin{abstract}
	Given a multiplicatively closed subset $S$ of the integers, there exist Structure Theorems for $LC$ modules over the localization $\mathbb{Z}\inv{S}$ (their category is denoted $LC_{\mathbb{Z}\inv{S}}M$) that are "similar" to those of $LCA$ groups.\\
	The most notable one is the 1st Theorem: Given $M \in LC_{\mathbb{Z}\inv{S}}M$, there exists a unique set of prime numbers $\Sigma$ (purely dependent on $S$) for which  $
		M \cong \mathbb{R}^n \times \sideset{}{'}\prod_{q \in \Sigma} \mathbb{Q}_p^{n_p} \times N$,
	where $(n, (n_p)_{p \in \Sigma})$ is a sequence of nonnegative integers and $N$ contains a compact open submodule $K$ such that $K/K_0$ is a topological module over $\prod_{ q \in \mathbb{P}\setminus{\Sigma}} \mathbb{Z}_q$.
	\noindent Just like for $LCA$ groups, it is also possible to characterize the locally compact, compactly generated modules over $\mathbb{Z}\inv{S}$, as well as their Pontryagin Duals (which then allows to conclude that any locally compact $\mathbb{Z}\inv{S}$-module is an inverse limit of modules within a specific family). These characterizations are given in the 2nd and 3rd Structure Theorems respectively.
	\noindent Furthermore, as an elementary consequence of the 1st Structure Theorem, one can obtain a full classification of locally compact vector spaces over $\mathbb{Q}$.

\end{abstract}

\maketitle

\clearpage
{
	\hypersetup{linkcolor=black}
	\tableofcontents
}

\clearpage

\pagenumbering{arabic}

\section{Introduction}
\label{introduction}

In a general algebraic setting, if $R$ is a commutative ring, $S$ a multiplicatively closed subset of $R$ and $G$ an abelian group, there is a bijection

\[
	\left\{
	\begin{aligned}
		 & R\inv{S}-\text{module}   \\
		 & \text{ structures on } G
	\end{aligned}
	\right\}
	\longleftrightarrow
	\left\{
	\begin{aligned}
		 & R-\text{module structures on } G                               \\
		 & \text{that are } S-\text{divisible and } S-\text{torsion-free}
	\end{aligned}
	\right\}
\]

However, in the case of $\mathbb{Z}$-modules, the structure is always unique. This means that an abelian group admits a $\mathbb{Z}\inv{S}$-module structure if and only if it is $S$-divisible and $S$-torsion-free, which is a completely intrinsic characterization. \\
Furthemore, if an abelian group $A$ is both a $R$-module and a $\mathbb{Z}\inv{S}$-module at the same time, it is automatically a bimodule (this is for example helpful in Lemma \ref{compactly generated elliptic and totally disconnected module} and Proposition \ref{guaranteed sigma compactness}). Another useful feature is that any abelian group homomorphism between $\mathbb{Z}\inv{S}$-modules is automatically $\mathbb{Z}\inv{S}$-linear.  \\
Within the context of $LCA$ groups, things complicate slightly: Being a $\mathbb{Z}\inv{S}$-module does not seem to guarantee that the scalar multiplication is continuous. Whether this happens or not can be summarized in the following way: \begin{itemize}
	\item Given a $LCA$ group $G$ that is $S$-divisible and $S$-torsion-free, it is sufficient to check if the map \begin{align*}
		      S \times G & \to G                \\
		      (s, g)     & \mapsto \frac{1}{s}g
	      \end{align*}
	      is continuous.
	\item Because $S$ is discrete, the above is equivalent to saying that for each $s \in S$, $g \mapsto \frac{1}{s}g$ is continuous, which is the same as $g \mapsto sg$ being an automorphism of $G$ (because of local compactness and Hausdorffness of $G$, this is also equivalent to saying that it is a proper map).  \\
\end{itemize}

If $\mu$ is a Haar measure on $G$ and $s \in S$, multiplication by $s$ being an automorphism implies that there exists some $\lambda(s) > 0$ such that $\mu(sE) = \lambda(s)\mu(E)$ for any Borel set $E \subset G$. This means that multiplication by $s$ "uniformly scales" subsets of $G$. \\
Even though it is not entirely clear what restrictions such property imposes on the topology of $G$, one sufficient condition for continuity of scalar multiplication to hold is $\sigma$-compactness, since it allows us to apply the \hyperref[Open mapping thm]{Open Mapping Theorem}. \\
Despite a lack of a satisfactory intrinsic characterization of locally compact topological $\mathbb{Z}\inv{S}$-modules, they nonetheless admit Structure Theorems (Theorems \ref{1st Structure Theorem}, \ref{2nd Structure Theorem} and \ref{3rd Structure Theorem}) which hold several similarities to those of $LCA$ groups (and in a certain sense, they impose restrictions on the structure of those groups). The objective of this article is to prove them, while setting up relevant context before doing so. \\
From here on, any time one talks about a locally compact $\mathbb{Z}\inv{S}$-module, it is implicitly assumed that the module action is continuous. All groups and modules are taken to be topological and Hausdorff throughout. \\

\noindent The body of this paper is divided as follows:
\begin{itemize}
	\item Section \ref{Preliminaries} introduces some definitions and well known but essential results throughout subsections \ref{Annihilators and closed subgroups}-\ref{LC modules over the profinite integers}. On subsection \ref{Adaptation of results and some extentions of definitions}, the notions of module of \hyperref[Definition of module of Lie type]{ Lie type } and \hyperref[Definition of NSS and compactly generated]{ NSS } property are introduced for locally compact modules over a topological ring, and some results that are typically applied to topological groups (like the Open Mapping Theorem) are adapted to suit the context of this article.
	\item On Section~\ref{Localizations of integers}, we first start by defining the "ring of Adeles" associated to a localization $\mathbb{Z}\inv{S}$ of the integers, as well as the "finite Adeles" (Definition \ref{generalized adele rings}). It is proved in Subsection \ref{Generalized Adele rings} that any locally compact module over the latter is a \hyperref[Restricted product definition]{restricted product} of finite dimensional vector spaces over $\mathbb{Q}_p$, as $p$ ranges over a family of prime numbers $\Sigma$ uniquely associated to $S$, before moving to the next subsection.
	\item	      Subsection \ref{modules over localizations of integers} is entirely dedicated to proving the Structure Theorems for $LC$ $\mathbb{Z}\inv{S}$-modules, and has a small corollary attached to it, which gives a  classification of locally compact $\mathbb{Q}$-vector spaces.
	\item Subsection \ref{Modules of Lie type} is rather small, and asks some questions about classifying $LC$ $\mathbb{Z}\inv{S}$-modules of Lie type.
\end{itemize}

\section{Preliminaries}
\label{Preliminaries}

\begin{definition}
	Let $G$ and $G'$ be topological groups. $G$ and $G'$ are said to be locally isomorphic if there exist open neighbourhoods $U \subset G$ and $V \subset G'$ of $e_G$ and $e_{G'}$ respectively, and a homeomorphism $f\colon U \to V$ with the following properties:\begin{itemize}
		\item For any $x, y \in U$ such that $xy \in U$, $f(xy) = f(x)f(y)$.
		\item For any $x \in U$ such that $\inv{x} \in U$, $f(\inv{x}) = \inv{f(x)}$.
		\item Furthermore, if $G$ and $G'$ are abelian and also left topological modules over a topological ring $R$, it is required that whenever $x \in U$ and $r \in R$ are such that $rx \in U$, $f(rx) = rf(x)$.
	\end{itemize}

	In particular, a $LCA$-group $G$ is said to be of Lie type if it is locally isomorphic to $(\mathbb{R}^n, +)$, for some $n \in \mathbb{N}_0$.
\end{definition}

\begin{theorem}[Structure theorems]
	Let $G \in LCA$.
	\begin{enumerate}
		\item $G$ is isomorphic to $\mathbb{R}^n \times H$, for some $n \ge 0$ and $H$ containing a compact open subgroup.
		\item $G$ is compactly generated if and only if it is isomorphic to $\mathbb{R}^n \times \mathbb{Z}^k \times K$, for some $n, k \ge 0$ and $K$ a compact abelian group.
		\item $G$ is of Lie type if and only if it is isomorphic to $\mathbb{R}^n \times (\unitcircle)^k \times D$, where $n, k \ge 0$ and $D$ is discrete.
	\end{enumerate}
\end{theorem}

\noindent Because $\dual{\mathbb{R}^n} \cong \mathbb{R}^n$, $\dual{\mathbb{Z}^k} \cong (\unitcircle)^k$ and compact abelian groups are dual to discrete abelian groups, the second and third structure theorems allow us to conclude that compactly generated $LCA$-groups are in correspondence with $LCA$-groups of Lie type through Pontryagin Duality. \\
Chapter 4 of \cite{DeitmarEchterhoff2014} is dedicated to proving these theorems. \\

\subsection{Annihilators and closed subgroups}
\label{Annihilators and closed subgroups}

Let $G \in LCA$. Its Pontryagin Dual is denoted $\dual{G}$.

\begin{definition}
	If $S \subset G$, we define $A(S) := \{\Chi \in \dual{G}: \ \Chi|_S = 1\}$ and for $X\subset \dual{G}$, $\orth{X} := \{g \in G: \chi(g) = 1 \ \text{for all} \ \chi \in X\}$. These sets are called the annihilator of $S$ and orthogonal complement of $X$, respectively. It's easy to check that $A(S) = A(<S>) = A(\closure{<S>})$ for any $S \subset G$. \\
	One can note that for $X \subset \dual{G}$, $\orth{X} = \invimage{\eta_G}{A(X)}$, by \hyperref[Duality thm]{Pontryagin Duality}. So anything that holds for annihilators will also hold for orthogonal complements.
\end{definition}

\begin{prop}
	\label{annihilator isomorphisms}
	If $G \in LCA$ and H is a closed subgroup of G:
	\begin{enumerate}
		\item $\dual{G/H} \cong A(H)$
		\item $\orth{A(H)} = H$
		\item $\dual{G}/A(H) \cong \dual{H}$
	\end{enumerate}
\end{prop}

The full proof won't be given since it is standard, but the way one arrives at $1.$ and $3.$ is by \hyperref[dual morphism]{dualizing} the projection $\pi \colon G \to G/H$ and inclusion $\iota \colon H \to G$ respectively (and then checking the respective isomorphisms by doing some analysis). $2.$ is a consequence of $1.$ along with the fact that characters of $LCA$ groups separate their elements.

\begin{corollary}
	\label{subgroup - annihilator topological relations}
	If $H \le G$ is closed, then
	\begin{itemize}
		\item H is compact if and only if A(H) is open (co-discrete).
		\item H is co-compact if and only if A(H) is discrete
	\end{itemize}
	This is concluded by the fact that compact abelian groups are dual to discrete abelian groups, along with the above Proposition.
\end{corollary}

\begin{lemma}
	If $\{S_i\}_{i \in I}$ is a family of subsets of G, then
	\[A(\bigcup_i S_i) = \bigcap_i A(S_i)\]
	This is easy to check by the definition of annihilator.
\end{lemma}

\begin{corollary}
	\label{ann of sums}
	If $\{H_i\}_{i \in I}$ is a family of subgroups of G, then \[
		A(\closure{\sum_i H_i}) = A(\sum_i H_i) = A(<\bigcup_i H_i>) = A(\bigcup_i H_i) = \bigcap_i A(H_i)\]
	Furthermore, if they are closed, then \[
		A(\bigcap_i H_i) = \closure{\sum_i A(H_i)}
	\]
\end{corollary}
\begin{proof}

	\[
		\bigcap_i H_i = \bigcap_i \orth{A(H_i)} = \orth{(\closure{\sum_i A(H_i)})}
	\]
	by what was just seen, and because the $H_i$'s are closed. Applying annihilators to both sides, we get \[
		A(\bigcap_i H_i) = A(\orth{(\closure{\sum_i A(H_i)})}) = \closure{\sum_i A(H_i)}
	\]
\end{proof}

\subsection{The subgroup of compact elements}
\label{subgp of compact elements}
The notion of elliptic $LCA$ groups and their duality with totally disconnected $LCA$ groups is helpful throughout this article.

\begin{definition}
	If $x \in G$, x is said to be a compact element if $\closure{<x>}$ is compact. C(G) will denote the set of all compact elements of G. It is also equal to the union of all compact subgroups of G, because of Hausdorffness.
\end{definition}

\begin{theorem}
	\label{ann of connected component}
	Given $G \in LCA$,
	\[
		A(G_0) = C(\dual{G}),
	\]
	where $G_0$ is the connected component at the identity of $G$.
\end{theorem}

\noindent To prove this, the following lemma is useful:
\begin{lemma}
	Let $G \in LCA$ be totally disconnected. Then for any $\psi \in \dual{G}$, there exists a compact open subgroup H of G such that $\psi(H) = {1}$.
\end{lemma}

\begin{proof}
	Let $U \subset \unitcircle$ an open neighbourhood of 1 that does not contain any nontrivial subgroups. Then $\invimage{\psi}{U}$ is an open neighbourhood of 0 in G. By Theorem \ref{Locally compact t.d.}, there exists a compact open subgroup H of G such that $H \subset \invimage{\psi}{U}$. \\
	$\image{\psi}{H}$ will then be a subgroup of $\unitcircle$ contained in U, therefore it is trivial.
\end{proof}

\begin{proof}[Proof of Theorem \ref{ann of connected component}]
	If $\chi \in A(G_0)$, then by the construction done in the first part of \ref{annihilator isomorphisms}, we get a character $\psi\colon G/G_0 \to \unitcircle$ given by $\psi(g + G_0) = \chi(g)$. $G/G_0$ is totally disconnected by Lemma \ref{conn.component}, so there exists some $X \le G/G_0$ compact and open such that whenever $g + G_0 \in X$, $\chi(g) = \psi(g + G_0) = 1$. Therefore, $H = \invimage{\pi}{X}$ is an open subgroup of G on which $\chi$ acts trivially. \\
	Conversely, if $\chi$ acts trivially in some open subgroup H, then since by Lemma \ref{conn.component}, $G_0$ is necessarily contained in H, it follows that $\chi \in A(G_0)$. This together with Corollary  \ref{subgroup - annihilator topological relations} implies
	\[
		A(G_0) = \bigcup_{\mathclap{\substack{H \le G \\ H open}}} A(H) = \bigcup_{\mathclap{\substack{X \le \dual{G} \\ X \ compact}}} X = C(\dual{G})\]
\end{proof}

\noindent Note that there would be no reason to assume that $C(\dual{G})$ is closed since it is an arbitrary union of closed sets. This result also implies, by Pontryagin duality, that $C(G)$ is always a closed subgroup of G.  \\

\begin{corollary}
	\label{elliptic and t.d. duality}
	For $G \in LCA$, $C(G) = G$ if and only if $\dual{G}$ is totally disconnected, and $C(G) = 0$ if and only if $\dual{G}$ is connected. \\
	Furthermore, $\dual{G/C(G)} \cong A(C(G)) = (\dual{G})_0$, which is connected, therefore $G/C(G)$ has no nontrivial compact subgroups. This can be seen as an analogue of the fact that for any discrete abelian group $D$, $D/T(D)$ is torsion-free.  \\
	In particular, if $K$ is a compact abelian group, $A(K_0) = T(\dual{K})$, the torsion subgroup of $\dual{K}$, which means that, $K$ is connected if and only if $\dual{K}$ is torsion-free and $K$ is totally disconnected if and only if $\dual{K}$ is a torsion abelian group. \\
\end{corollary}

\begin{corollary}[Definition of elliptic]
	\label{union of compact open sbgps}
	$C(G) = G$ if and only if $G$ is a union of compact and open subgroups. In this situation, $G$ is called elliptic (local compactness will be assumed by default in the definition of elliptic).
\end{corollary}

\begin{proof}
	If $C(G) = G$, then $\dual{G}$ is totally disconnected, therefore \[
		\bigcap_{\mathclap{\substack{X \le \dual{G} \\ X \ open \ and \ compact}}} X  = \{1\},
	\] which by Corollary \ref{ann of sums} implies \[
		\closure{\sum_{\mathclap{\substack{H \le G \\ H  \ open \ and \ compact}}} H} = G.
	\]
	But the sum of all compact and open subgroups of $G$ is an open subgroup, (since it is the sum of a non-empty open set with something else) and therefore closed. At the same time, this sum is also equal to the union of all compact and open subgroups of G: \\
	If $x$ belongs to that sum, then there exist compact open subgroups $H_1, \dots, H_n$ such that $x \in H_1 + \dots + H_n$ and $H_1 + \dots + H_n$ is also a compact and open subgroup, so x is in the mentioned union. In conclusion, \[
		G \quad = \quad \sum_{\mathclap{\substack{H \le G \\ H \ open \ and \\ compact}}} H \quad = \quad \bigcup_{\mathclap{\substack{H \le G \\ H \ open \ and \\ compact}}} H
	\]
\end{proof}

\begin{corollary}
	\label{compact subset contained in compact open subgroup}
	If $G \in LCA$ is elliptic, any compact subset of $G$ is contained in a compact open subgroup. In particular, any compactly generated subgroup of $G$ has compact closure.
\end{corollary}

\begin{proof}
	Let $K \subset G$ be compact. For each $x \in K$, let $H_x$ be a compact open subgroup of $G$ containing $x$. Then by compactness of $K$, \[
		K \subset \bigcup_{i=1}^m H_{x_i} \subset H = \sum_{i=1}^m H_{x_i},
	\]
	for $x_1, \dots, x_m \in K$. So $H$ is a compact open subgroup of $G$ containing $K$.
\end{proof}

\subsection{Pontryagin Dual of a module}
\label{Pontryagin Dual of a module}

Let R be a topological ring and $M \in LC_R M$ (that is, M is a locally compact left topological $R$-module). Then for $\chi \in \widehat{M}$ and $r \in R$, we can define \begin{align*}
	\chi^r\colon & M \ \to \unitcircle \\
	             & m \mapsto \chi(rm)
\end{align*}
$\chi^r$ is still an element of $\dual{M}$ because M is a topological R-module.
It is easy to check that for all $\chi, \psi \in \dual{M}$ and $r, r' \in R$: \hfill \begin{enumerate}
	\item $\chi^{1_R} = \chi$
	\item $(\chi \psi)^r = \chi^r \psi^r$
	\item $\chi^{rr'} = {(\chi^r)}^{r'}$
	\item $\chi^{r + r'} = \chi^{r}\chi^{r'}$
\end{enumerate}
So $\dual{M}$ becomes a right (not necessarily topological) R-module, under multiplicative notation.

\begin{theorem}
	\label{pontryagin dual of a module}
	If R is locally compact, then $\dual{M}$ is a topological R-module with respect to the above construction.
\end{theorem}

\begin{proof}
	We need to show that the map \begin{align*}
		\dual{M} \times R \  & \to \dual{M}   \\
		(\chi, r)            & \mapsto \chi^r
	\end{align*}
	is continuous. We show this through net convergence: Let $(\chi_\alpha)_\alpha$ be a net  converging to  $\chi \in \dual{M}$ and $(r_\alpha)_\alpha$ converging to $r \in R$.

	\[
		\chi_\alpha^{r_\alpha}\chi^{-r} = \chi_\alpha^{r_\alpha}\chi^{-r - r_\alpha +r_\alpha} = (\chi_\alpha\inv{\chi})^{r_\alpha} \chi^{r_\alpha - r}.
	\]
	We want to show that the above expression converges to the trivial character uniformly on compact sets. Note that \ $\abs{(\chi_\alpha \inv{\chi})^{r_\alpha}\chi^{r_\alpha - r} - 1} \le \abs{(\chi_\alpha \inv{\chi})^{r_\alpha} - 1} \  + \  \abs{\chi^{r_\alpha - r} - 1}$.\\
	Let $V$ be a compact neighbourhood of $0_R$, $K$ be a compact subset of M and $\epsilon > 0$. Then $r + V$ is a compact neighbourhood of r and there exists $\beta_1$ such that $\alpha \ge \beta_1 \ \implies r_\alpha \in r+V$. For $\alpha \ge \beta_1$, we get
	\begin{align*}
		 & \sup_{m \in K} \abs{\chi_\alpha(r_\alpha m)\chi(-r_\alpha m) - 1} = \sup_{m \in K}\abs{\chi_\alpha(r_\alpha m) - \image{\chi}{r_\alpha m}} \\
		 & \le \sup_{m \in (r+V)K}\abs{\chi_\alpha(m) - \chi(m)}
	\end{align*}

	Since $\chi_\alpha \longrightarrow \chi$ and $(r+V)K$ is compact, there exists $\beta_2$ such that for $\alpha \ge \beta_2$, the last supremum becomes smaller than $\frac{\epsilon}{2}$. So we take $\beta$ to be some index that is bigger than both $\beta_1$ and $\beta_2$. \\
	Let $f\colon R \times M \to M$ be the R-module multiplication in M. Then $f\arrowvert_{(r+V)\times K}$ is uniformly continuous by Theorem \ref{Guaranteed uniform continuity}, so for any neighbourhood U of $0$ in M, there exists an open set $A \subset R\times M$ containing $(0_R, 0)$ such that whenever $(r_1, m_1), (r_2, m_2) \in (r+V)\times K$ and $(r_1 - r_2, m_1 - m_2) \in A$, $r_1m_1 - r_2m_2 \in U$. \\
	In particular, A will contain $W\times \{0\}$, for some neighbourhood of $0_R$, W.\\
	Let $\tau_1$ be such that $\alpha \ge \tau_1 \implies r_\alpha - r \in W$ and $\tau_2$ such that $\alpha \ge \tau_2 \implies r_\alpha \in r+V$. Then if $\tau$ is some upper bound of $\tau_1, \tau_2$, we get that for $\alpha \ge \tau$ and any $m \in K$, \[(r_\alpha, m), (r, m) \in (r+V)\times K, \ (r_\alpha - r, m - m) = (r_\alpha - r, 0) \in W\times \{0\} \subset A.
	\]
	So for $\alpha \ge \tau$, we get that $r_\alpha m - rm \in U$, for any $m \in K$.
	In particular, let $U = \invimage{\chi}{\{z \in \unitcircle: \abs{z - 1} < \frac{\epsilon}{2}\}}$.
	Then \[\sup_{m \in K}\abs{\chi(r_\alpha m) - \chi(rm)} = \sup_{m \in K}\abs{\chi(r_\alpha m - rm) - 1} < \frac{\epsilon}{2}\] for $\alpha \ge \tau$. \\
	Let $\sigma \ge \beta, \tau$. Then for $\alpha \ge \sigma$, we get that \begin{align*}
		\sup_{m \in K}\abs{\chi_\alpha^{r_\alpha}(m)\chi^{-r}(m) - 1} & \le  \sup_{m \in K}\abs{\chi_\alpha(r_\alpha m) - \image{\chi}{r_\alpha m}} + \sup_{m \in K}\abs{\chi(r_\alpha m) - \chi(rm)} \\
		                                                              & < \frac{\epsilon}{2} + \frac{\epsilon}{2} = \epsilon
	\end{align*}
	In other words, $\chi_\alpha^{r_\alpha}$ converges uniformly on compact sets to $\chi^r$, so $\dual{M}$ is a topological R-module.
\end{proof}

\begin{prop}
	Let $M \in LC_RM$. The Pontryagin map $\eta_M$ is R-linear, therefore $M$ and $\dual{\dual{M}}$ are isomorphic as topological R-modules. Besides that, if $f\colon M \to N$ is a continuous R-linear homomorphism, then so is $\widehat{f}$ (see Definition \ref{dual morphism})
\end{prop}

\begin{proof}
	Let $r \in R$, $m \in M$ and $\chi \in \dual{M}$. \[
		^r\eta_M(m)(\chi) = \eta_M(m)(\chi^r)= \chi^r(m) = \chi(rm) = \eta_M(rm)(\chi)\]
	In other words, $^r\eta_M(m) = \eta_M(rm)$.  \\

	If $f\colon M \to N$ is R-linear, then given $\psi \in \dual{N}$, $m \in M$ and $r \in R$, \[
		\widehat{f}(\psi^r)(m) = \psi^r\circ f (m) = \psi(rf(m)) = \psi(f(rm)) = \psi \circ f (rm) = (\widehat{f}(\psi))^r (m)
	\]
	So $\widehat{f}(\psi^r) = (\widehat{f}(\psi))^r$

\end{proof}

Therefore, we get a contravariant functor $F_R$ from $LC_RM$ to $LC M_R$.
Since continuous R-linear maps are in particular continuous group homomorphisms, $\eta$ will still make the respective diagrams commute, so $\eta$ will still be a natural isomorphism between the identity functor in $LC_RM$ and $F_{R^{op}} \circ F_R$. As such, we get an equivalence of categories \[
	LC_RM \cong (LC M_R)^{op} \cong (LC_{R^{op}}M)^{op}\]

\noindent This implies that we are also allowed to swap direct and inverse limits when talking about modules, and that $F_R$ commutes with finite products in general, up to isomorphism. It also does the same thing with direct sums of discrete left $R$-modules and products of compact right $R$-modules. \\

\begin{prop}
	\label{annihilator of submodule}
	If $M \in LC_RM$ and $S \subset M$ is closed under scalar multiplication, then $A(S)$ is a closed submodule of $\dual{M}$.
\end{prop}

\begin{proof}
	We already know it is a closed subgroup. Let $\chi \in A(S)$ and $r \in R$. Given $s \in S$, $\chi^r(s) = \chi(rs) = 1$, since $rs \in S$. Therefore, $\chi^r \in A(S)$.
\end{proof}

So annihilators will establish a correspondence between closed R-submodules of M and closed R-submodules of its dual, and the same isomorphisms from Proposition \ref{annihilator isomorphisms} hold.

\begin{corollary}
	$M_0$ and $C(M)$ are always closed R-submodules of M.
\end{corollary}

\begin{proof}
	Combining Theorem \ref{ann of connected component} and Proposition \ref{annihilator of submodule}, it suffices to show that $M_0$ is a closed under scalar multiplication. \\
	Let $r \in R$. Then $rM_0$ is a connected subgroup of $M$, so it must be contained in $M_0$, which proves the statement.
\end{proof}

It is easy to see that if $R$ is an elliptic ring, then because it is also assumed to be unital, any topological module over it will be elliptic. In the locally compact case, Pontryagin Duality along with Lemma \ref{elliptic and t.d. duality}  allows us to conclude that any $LC$ topological module over $R$ is both elliptic and totally disconnected.

\subsection{LC modules over the profinite integers}
\label{LC modules over the profinite integers}

Locally compact modules over \hyperref[definition of profinite integers]{the profinite integers} $\hat{\mathbb{Z}}$ can be intrinsically characterized, and this will be useful for Lemma \ref{compactly generated elliptic and totally disconnected module}. A very useful fact about $\hat{\mathbb{Z}}$ is that the integers are a dense subring. This implies that any topological module structure over $\hat{\mathbb{Z}}$ is unique, and any closed subgroup of a topological $\hat{\mathbb{Z}}$-module is a submodule.

\begin{prop}
	Let $K$ be a compact abelian group. Then $K$ is totally disconnected (equivalent to being profinite as an abelian group in this case) if and only if $K \in LC_{\hat{\mathbb{Z}}}M$.
\end{prop}

\begin{proof}
	If $K \in LC_{\hat{\mathbb{Z}}}M$, it has to be totally disconnected by the final comment in the previous subsection. \\
	Conversely, if $K$ is totally disconnected, then $A = \dual{K}$ is a discrete torsion abelian group (Corollary \ref{elliptic and t.d. duality}). So given any $x = \sum_{i=0}^{+\infty} c_ii! \in \hat{\mathbb{Z}}$ and $a \in A$, the expression $(\sum_{i=0}^{+\infty}c_i i!)a$ is well defined, because $a$ has finite order. This allows us to define a natural $\hat{\mathbb{Z}}$-module structure on $A$. The scalar multiplication \begin{align*}
		f\colon \hat{\mathbb{Z}} \times A & \to A      \\
		(x, a)                            & \mapsto xa
	\end{align*}
	being continuous is equivalent to it being locally constant, since $A$ is discrete. But given $(x, a) \in \hat{\mathbb{Z}} \times A$ with $na = 0$ for some $n \in \mathbb{N}$, we can choose $U = (x + n\hat{\mathbb{Z}}) \times \{a\}$, and we clearly have that $f(u) = xa$, for any $u \in U$. So if $A$ is torsion, it admits a topological $\hat{\mathbb{Z}}$-module structure, which implies $K \in LC_{\hat{\mathbb{Z}}}M$ as well by Theorem \ref{pontryagin dual of a module} . \\

\end{proof}

\begin{theorem}
	\label{characterization of modules over profinite integers}
	Let $G \in LCA$. Then $G \in LC_{\hat{\mathbb{Z}}}M$ if and only if $G$ is elliptic and totally disconnected.
\end{theorem}

\begin{proof}
	The first implication is a result of $G$ and  $\dual{G}$ being both elliptic. \\
	So suppose $G$ is elliptic and totally disconnected. By what was seen in Corollary \ref{union of compact open sbgps}, $G$ is a union of compact open subgroups, and all of those will also be totally disconnected. So in this case, $G$ is a union of profinite open subgroups $(P_i)_{i \in I}$. For each $i \in I$, let $f_i\colon \hat{\mathbb{Z}} \times P_i \ \to P_i \subset G$ the topological $\hat{\mathbb{Z}}$-module structure in $P_i$. \\
	$f_i|_{\hat{\mathbb{Z}}\times (P_j\cap P_i)} = f_j|_{\hat{\mathbb{Z}}\times (P_j\cap P_i)}$ for all $i, j \in I$ because the topological $\hat{\mathbb{Z}}$-module strucures are all unique. Since $(\hat{\mathbb{Z}} \times P_i)_{i \in I}$ is an open cover of $\hat{\mathbb{Z}} \times G$, we can extend to a unique continuous function $f\colon \hat{\mathbb{Z}} \times G \to G$ such that $f|_{\hat{\mathbb{Z}}\times P_i} = f_i$ for all $i \in I$. \\
	We note now that $f$ gives $G$ a module strucure over $\hat{\mathbb{Z}}$ given by $x \cdot g = f(x, g)$ (and the action is continuous, since $f$ is): \\
	Let $x, y \in \hat{\mathbb{Z}}$, $g, h \in G$ with $g \in P_i$ and $h \in P_j$. \begin{enumerate}
		\item $1 \cdot g = f_i(1, g) = g$.
		\item $(x + y)\cdot g = f_i(x + y, g) = f_i(x, g) + f_i(y, g) = x\cdot g + y\cdot g$
		\item $(xy) \cdot g = f_i(xy, g) = f_i(x, f_i(y, g)) = x\cdot(y\cdot g)$
		\item $x\cdot (g + h) = f_{ij}(x, g + h)$, where $P_{ij} = P_i + P_j$. \\
		      $f_{ij}(x, g + h) = f_{ij}(x, g) + f_{ij}(y, g) = x\cdot g + y\cdot g$, because $x, y \in P_i + P_j$.
	\end{enumerate}

	So if $G$ is elliptic and totally disconnected, it admits a topological module structure over $\hat{\mathbb{Z}}$.\\
\end{proof}

\subsection{Adaptation of results and some extentions of definitions}
\label{Adaptation of results and some extentions of definitions}

Before looking into the Structure Theorems, we will adapt some results from chapter 4 of \cite{DeitmarEchterhoff2014} to the case of (left) topological modules over a topological ring $R$ in this subsection (the results being adapted can be found in pages 98-101 of \cite{DeitmarEchterhoff2014}). Furthemore, some extra definitions will also be given. \\
All homomorphisms are assumed to be $R$-linear, unless otherwise stated.

\begin{lemma}[Lemma 4.2.7]
	\label{continuous section}
	Let $M$ and $N$ be topological $R$-modules, and $q\colon M \to N$ a continuous surjective homomorphism. Suppose additionally that there exists a continuous section $s$ of $q$, that is, a continuous homomorphism $s: N \to M$ such that $q\circ s = Id_N$. \\
	Then $M$ is isomorphic to $ker(q) \times N$.
\end{lemma}

\begin{proof}
	By Lemma 4.2.7 of \cite{DeitmarEchterhoff2014}, the map $\phi\colon M \to ker(q) \times N$ given by $\phi(x) = (x - s\circ q(x), q(x))$ is an isomorphism of topological groups. Therefore, it suffices to check that it is also $R$-linear: \\
	\begin{align*}
		 & \phi(rx) = (rx - s\circ q(rx), q(rx)) = (rx - r(s\circ q(x)), rq(x)) = \\
		 & = r(x - s\circ q(x), q(x)) = r\phi(x)
	\end{align*}
\end{proof}

\begin{corollary}[Example 4.2.8 of \cite{DeitmarEchterhoff2014}]
	\label{free quotient implies split}
	If $M$ is a topological $R$-module such that there exists a continuous surjective homomorphism $q\colon M \to R^n$ for some $n \in \mathbb{N}$, then $M \cong ker(q) \times R^n$. Furthermore, if $R$ is discrete, we can replace $R^n$ with any free left $R$-module (with the discrete topology).
\end{corollary}

\begin{proof}
	For each $i \in \{1, \dots, n\}$, we can choose $x_i \in M$ such that $q(x_i) = e_i$, where $e_i$ is the i-th basis vector of $R^n$. By letting $s\colon R^n \to M$ be the unique homomorphism such that $s(e_i) = x_i$, $s$ is a continuous section of $q$, therefore we can apply the above lemma (continuity is granted by the fact that $M$ is a topological $R$-module). \\
	If $R$ is discrete and $F$ is a free $R$-module, we can construct a section of $q$ in the exact same way, and its continuity will be automatically given.
\end{proof}

\begin{theorem}[Open Mapping Theorem]
	\label{Open mapping thm}
	Let $G$ and $H$ be locally compact groups, with $G$ $\sigma$-compact and suppose $\phi\colon G \to H$ is a continuous surjective homomorphism. Then $\phi$ is an open map.
\end{theorem}
\noindent Proof can be found in Theorem 4.2.9 of \cite{DeitmarEchterhoff2014}. In particular, this result holds if we replace $G$ and $H$ by topological $R$-modules and $\phi$ by an $R$-linear homomorphism, under the same conditions.

\begin{corollary}[Corollary 4.2.11 of \cite{DeitmarEchterhoff2014}]
	\label{sigma compact internal direct sum}
	Suppose that $M$ is a locally compact $R$-module and $N, N'$ are closed, $\sigma$-compact submodules of $M$ such that $N + N' = M$ and $N \cap N' = 0$. Then the map $\phi\colon N \times N' \to M$ that sends $(n, n')$ to $n + n'$ is an isomorphism of topological $R$-modules.
\end{corollary}

\begin{proof}
	$\phi$ is continuous, $R$-linear and bijective, by hypothesis. Since $N \times N'$ is locally compact and $\sigma$-compact, $\phi$ is also open by the Open mapping Theorem.
\end{proof}

\begin{lemma}[Lemma 4.2.14 of \cite{DeitmarEchterhoff2014}]
	\label{lemma 4.2.14}
	Let $M \in LC_RM$ and $L, Q$ closed submodules of $M$ such that: \begin{itemize}
		\item $L$ and $Q$ are $\sigma$-compact.
		\item $L + Q$ is open in $M$ and $L\cap Q = 0$.
		\item $Q$ is algebraically injective as an $R$-module.
	\end{itemize}
	Then there exists a closed submodule $N$ of $M$ such that $L \subset N$, $L$ is open in $N$, and such that the map \begin{align*}
		\phi\colon & N \times Q \to M     \\
		           & (n, q) \mapsto n + q
	\end{align*}
	is an isomorphism of topological $R$-modules.
\end{lemma}

\begin{proof}
	By the previous lemma, $L \times Q \cong L + Q$ as topological $R$-modules, so we can consider the map $\psi\colon L + Q \to Q$ that sends $l + q$ to $q$, and it is continuous since it is the composition of $(l, q) \mapsto q$ with an isomorphism. Because $Q$ is algebraically injective, we can extend $\psi$ to a homomorphism $\overline{\psi}\colon M \to Q$. Since $\overline{\psi}$ is continuous on $L + Q$ and is in particular a group homomorphism, $\overline{\psi}$ is continuous. \\
	The inclusion $\iota\colon Q \to M$ is a section of $\overline{\psi}$, because $\overline{\psi}|_Q = \psi|_Q = Id_Q$, therefore by lemma \ref{continuous section}, $M \cong \ker(\overline{\psi}) \times Q$. \\
	Choose $N = \ker(\overline{\psi})$, and we get that $L \subset N$. Besides that, $N \cap (L + Q) = \ker(\overline{\psi}|_{L+Q}) = \ker(\psi) = L$. Therefore, $L$ is open in $N$.
\end{proof}

\begin{corollary}[Example 4.2.15 of \cite{DeitmarEchterhoff2014}]
	\label{Open injective submodule}
	If $M$ is a topological $R$-module and $Q$ is an open injective submodule of $M$, we can take $L = 0$ and apply the same reasoning of  Lemma \ref{lemma 4.2.14} to conclude that there exists a closed submodule $D$ of $M$ such that $M \cong Q \times D$ and such that $L = 0$ is open in $D$, that is, such that $D$ is discrete. \\
	In this particular case, by checking the above proof, it can be seen that local compactness of $M$ and $\sigma$-compactness of $Q$ are unnecessary, as a consequence of $L$ being trivial.
\end{corollary}

\begin{definition}
	\label{Definition of NSS and compactly generated}
	Let $M$ be a left topological $R$-module. $M$ is said to have the no small submodules property if there exists a neighbourhood $V$ of 0 such that the only submodule of $M$ contained in $V$ is the trivial one. \\
	We also define compactly generated in a straightforward way inspired by the definition for $LCA$ groups: $M$ is compactly generated if there exists a compact subset $K$ of $M$ such that $M = RK$.

\end{definition}

\begin{prop}
	\label{no small submodules implies dual is compactly generated}
	Let $R \in LCR$ and $M \in LC_RM$ such that $\dual{M}$ has no small submodules. Then $M$ is compactly generated.
\end{prop}

\begin{proof}
	Let $V$ be a neighbourhood of the identity in $\dual{M}$ that contains no nontrivial submodules. We can suppose without loss of generality that $V = V_{K, \epsilon}(1)$ for some $K \subset M$ compact neighbourhood of the identity and some $\epsilon > 0$, by Lemma \ref{neighbourhood basis on dual}.\\
	Let $N = RK$. $N$ is an open submodule of $M$, therefore also closed. At the same time, $A(N) \subset A(K) \subset V$, by definition of $V$. Since $A(N)$ is a submodule of $\dual{M}$ contained in $V$, it follows by hypothesis that $A(N) = 1$, and therefore $RK = N = M$.
\end{proof}

\noindent There is a partial converse to the above proposition:
\begin{prop}
	\label{fin generated implies no small submodules on dual}
	Let $M \in LC_RM$ be finitely generated. Then $\dual{M}$ has no small submodules. In particular, $\dual{R}^m \cong \dual{R^m}$ has this property for any $m \in \mathbb{N}$ (both as a left and right topological $R$-module).
\end{prop}

\begin{proof}
	Let $\epsilon > 0$ such that $\{z \in \unitcircle: \ \abs{z - 1} < \epsilon\}$ contains no nontrivial subgroups of $\unitcircle$ and
	$M = Rm_1 + \dots + Rm_n$ for $m_1, \dots, m_n \in M$.\\
	For each $i \in \{1, \dots, n\}$, let $V_i = V_{\{m_i\}, \frac{\epsilon}{n}}$.
	Consider $V = V_1 \cap \dots \cap V_n$. $V$ is an open neighbourhood of the identity in $\dual{M}$, and if $\chi \in V$ such that $\chi^r \in V$ for all $r \in R$, then given $m = r_1m_1 + \dots + r_nm_n$, we get
	\begin{align*}
		 & \abs{\chi(m) - 1} = \abs{\chi(r_1m_1)\dots \chi(r_nm_n) - 1} \le \abs{\chi(r_1m_1) - 1} + \dots + \abs{\chi(r_nm_n) - 1} = \\
		 & = \abs{\chi^{r_1}(m_1) - 1} + \dots + \abs{\chi^{r_n}(m_n) - 1} < n \frac{\epsilon}{n} = \epsilon
	\end{align*}
	As such, $\abs{\chi(m) - 1} < \epsilon$ for all $m \in M$, therefore $\chi(M)$ is a subgroup of $\unitcircle$ contained in $\{z: \abs{z - 1} < \epsilon\}$, which implies $\chi(M) = 1$ by definition of $\epsilon$. \\
	In other words, $V$ is a neighbourhood of the identity in $\dual{M}$ that contains no nontrivial submodules.
\end{proof}

\noindent There is another definition which intends to generalize the notion of $LCA$ groups of Lie type:
\begin{definition}
	\label{Definition of module of Lie type}
	$R_R$ is defined as the right topological module structure on $R$ over itself.\\
	Given $M \in LC_RM$, we say that $M$ is of Lie type if there exists a local isomorphism of modules between $M$ and $\dual{R_R}^n$, for some $n \ge 0$.
\end{definition}

\noindent This definition emcompasses the usual notion of Lie type for groups, because $\dual{\mathbb{Z}}^n \cong (\unitcircle)^n$ and $(\unitcircle)^n$ is locally isomorphic to $\mathbb{R}^n$.

\begin{prop}
	\label{Lie type implies NSSM}
	If $M \in LC_RM$ is of Lie type, then $M$ has no small submodules.
\end{prop}

\begin{proof}
	Let $\phi\colon V \to U$ be a local isomorphism, where $U \subset M$ and $V \subset \dual{R_R}^n$ are open neighbourhoods of each respective identity. Let $W \subset \dual{R_R}^n$ be an open neighbourhood of the identity that contains no nontrivial submodules and let $V' = \invimage{\phi}{W \cap U}$. If $N$ is a submodule of $M$ contained in $V'$, then by the fact that $\phi$ is local isomorphism, $\phi(N)$ is a submodule of $\dual{R_R}^n$ contained in $W$, which means $\phi(N) = 0$. Because $\phi$ is injective, this implies $N = 0$ and proves the result.
\end{proof}

\noindent The converse of this Proposition is false: For example, for a prime $p$, $\mathbb{Q}_p$ has no small $\mathbb{Z}[1/p]$-submodules, but it cannot be of Lie type because $\dual{\mathbb{Z}[1/p]}$ is a connected abelian group while $\mathbb{Q}_p$ has a neighbourhood basis consisting of compact open subgroups.

\section{Structure Theorems for LC modules over localizations of the integers}
\label{Localizations of integers}

\subsection{Rings of "Adeles" associated to localizations}
\label{Generalized Adele rings}

The rings defined below, among other things, have no compact open $\mathbb{Z}\inv{S}$-submodules and are self-dual. They are quite helpful for the Structure Theorems over $\mathbb{Z}\inv{S}$ because they allow to "separate" any $LC$ $\mathbb{Z}\inv{S}$-module into a component that can be described numerically and a component that has a compact open submodule.
\begin{definition}
	\label{generalized adele rings}
	Let $\Sigma$ be a set of prime numbers, $S$ the multiplicative submonoid of $\mathbb{N}$ generated by $\Sigma$ and $\mathbb{Z}\inv{S} \subset \mathbb{Q}$ the localization of $\mathbb{Z}$ with respect to $S$.
	\begin{itemize}
		\item The ring of Adeles associated to $S$, denoted by $\mathbb{A}_{\mathbb{Z}\inv{S}}$, is the restricted product $\mathbb{R} \times \sideset{}{'}\prod_{q \in \Sigma}(\mathbb{Q}_q, \mathbb{Z}_q)$.
		\item The ring of finite Adeles associated to $S$, denoted by $\mathbb{A}_{\mathbb{Z}\inv{S}, f}$, is the above restricted product without the $\mathbb{R}$ component.
		\item The ring of $\Sigma$-adic integers is the product $\mathbb{Z}_\Sigma := \prod_{q \in \Sigma} \mathbb{Z}_q$. It is in particular a quotient of $\hat{\mathbb{Z}}$ and is also a compact open subring of $\mathbb{A}_{\mathbb{Z}\inv{S}, f}$.
	\end{itemize}
	All of these are locally compact topological rings, and they will be quite important for locally compact $\mathbb{Z}\inv{S}$-modules. In particular, $\mathbb{A}_{\mathbb{Z}\inv{S}}$ will have an analogous role in $LC_{\mathbb{Z}\inv{S}}M$ that $\mathbb{R}$ has in $LCA$.
\end{definition}

\begin{example}
	If $\Sigma = \mathbb{P}$, the set of all primes, then $\mathbb{Z}\inv{S} = \mathbb{Q}$ and indeed, $\mathbb{A}_\mathbb{Q} = \prod_{q \in \Sigma}^{'}\mathbb{Q}_p$ is typically known as the ring of Adeles of the rational numbers, and the ring of $\mathbb{P}$-integers is $\prod_{q \ prime} \mathbb{Z}_q \cong \hat{\mathbb{Z}}$. \\
	If we instead take $\Sigma = \{p_1, \dots, p_n\}$ as a finite set of primes, then $\mathbb{Z}\inv{S} = Z[1/n]$, where $n = p_1 \dots p_n$, and the ring of $\Sigma$-adic integers is the product $\prod_{i = 1}^{n} \mathbb{Z}_{p_i} \cong \mathbb{Z}_n$ (the ring of n-adic integers). \\
\end{example}

\begin{lemma}
	\label{p-adic numbers are quotient of finite Adele ring}
	For each $p \in \Sigma$, $\mathbb{Q}_p$ is a quotient of $\mathbb{A}_{\mathbb{Z}\inv{S}, f}$. \\
	The obvious kernel of this quotient will be constantly denoted as $\prod_{q \in \Sigma \setminus \{p\}} \mathbb{Q}_q$.
\end{lemma}

\begin{proof}
	Consider the projection $\pi_p\colon \mathbb{A}_{\mathbb{Z}\inv{S}, f} \to \mathbb{Q}_p$. It is a surjective group homomorphism, so we only need to check continuity and openness at the identity. Because $\mathbb{A}_{\mathbb{Z}\inv{S}, f}$ is $\sigma$-compact, continuity suffices to prove both, by the Open Mapping Theorem. \\
	Given $n \in \mathbb{Z}$, $\invimage{\pi_p}{p^n\mathbb{Z}_p}$ can be written as \ $\sideset{}{'}\prod_{q \in \Sigma}X_q$, where \[
		X_q := \begin{cases*}
			(\mathbb{Q}_q, \mathbb{Z}_q), \ \text{if} \ q \neq p \\
			(p^n\mathbb{Z}_p, p^n\mathbb{Z}_p) \ otherwise,
		\end{cases*}
	\]
	since this restricted product is clearly a subset of $\mathbb{A}_{\mathbb{Z}\inv{S}, f}$. It is also open since it can be written as \[
		\bigcup_{\mathclap{\substack{g\colon \Sigma \to \mathbb{Z} \\ |supp(g)| < +\infty \\ g(p) = n}}} \quad \quad \left(\prod_{q \in \Sigma} q^{g(q)}\mathbb{Z}_q \right),
	\]
	which is an union of open sets. \\ Therefore $\pi_p$ is continuous, which proves the result. \\
\end{proof}

The relation between $\mathbb{A}_{\mathbb{Z}\inv{S}, f}$ and $\mathbb{Z}_\Sigma$ is somewhat similar to that between $\mathbb{Q}_p$ and $\mathbb{Z}_p$, for a fixed prime p, as will be seen in the two following propositions.

\begin{prop}
	\label{canonical localization relation}
	$\mathbb{A}_{\mathbb{Z}\inv{S}, f}$ can be algebraically seen as the localization of $\mathbb{Z}_{\Sigma}$ over $S$, as an abelian group.
\end{prop}

\begin{proof}
	Let $(x_q)_{q \in \Sigma} \in \mathbb{A}_{\mathbb{Z}\inv{S}, f}$. Since only a finite number of the $x_q$'s have negative $q$-adic valuation, we can consider \[
		s = \prod_{\mathclap{\substack{q \in \Sigma \\ v_q(x_q) < 0}}} q^{-v_q(x_q)} \in S,
	\] and we get that \[
		(x_q)_{q \in \Sigma} = \frac{1}{s}(y_q)_{q \in \Sigma},
	\]
	where each $y_q \in \mathbb{Z}_q$.
\end{proof}

\begin{prop}
	\label{guaranteed sigma compactness}
	Let $Q \in LC_{\mathbb{A}_{\mathbb{Z}\inv{S}, f}}M$ and let $C$ be any compact open subgroup of $C$. Then $Q$ is generated over $\mathbb{Z}\inv{S}$ by $C$.  \\
	In particular, $Q$ is $\mathbb{Z}\inv{S}$-compactly generated and therefore $\sigma$-compact by Proposition \ref{compactly generated over sigma compact ring}.
\end{prop}

\begin{proof}
	$Q$ is necessarily elliptic and totally disconnected, so for $C$ compact and open, we have that $\mathbb{Z}_{\Sigma}C \subset C$ (because $\mathbb{Z}_{\Sigma}$ is a quotient of $\hat{\mathbb{Z}}$ and topological $\hat{\mathbb{Z}}$-module structures are unique), while the quotient group $Q/C$ is discrete-torsion (since its dual is profinite). Therefore, given $q \in Q$, there exists $n \in N$ such that $nq \in C$. Note as well that by Proposition \ref{canonical localization relation}, we can write \[
		\frac{1_{\mathbb{Z}_{\Sigma}}}{n} = \frac{1}{s}(y_p)_{p \in \Sigma},
	\]
	for some $s \in S$ and where each $y_p \in \mathbb{Z}_p$. As such, \[
		q = (\frac{1_{\mathbb{Z}_{\Sigma}}}{n})nq = (\frac{1}{s}(y_p)_{p \in \Sigma})nq \in \frac{1}{s} (\mathbb{Z}_{\Sigma}C) \subset \frac{1}{s} C.
	\]
	Therefore, $Q$ is generated by $C$ as a $\mathbb{Z}\inv{S}$-module (reminder that the last manipulation makes sense because $Q$ is a $(\mathbb{Z}\inv{S}, \mathbb{A}_{\mathbb{Z}\inv{S}, f})$-bimodule).
\end{proof}

Proposition 4-13 of \cite{DeitmarEchterhoff2014} implies in particular that any locally compact topological vector space over $\mathbb{Q}_p$ is is isomorphic to a finite power of $\mathbb{Q}_p$. As will be shown in Theorem \ref{modules over finite adele ring}, locally compact modules over $\mathbb{A}_{\mathbb{Z}\inv{S}, f}$ can be classified by using this fact. \\
Before doing that however, it is useful to note that restricted products of locally compact vector spaces over fields of p-adic numbers do not depend on the choice of compact open subgroups, by the following theorem.

\begin{theorem}
	\label{invariance of choice of compact open subgroup for p-adic vector spaces}
	Let $(Q_q)_{q \in \Sigma}$ be a family of $LCA$ groups such that each $Q_q$ is a topological $\mathbb{Q}_q$-vector space, and let $Z_q$ be a compact open subgroup of $Q_q$ for each $q \in \Sigma$. Then there exists a sequence $(n_q)_{q \in \Sigma}$, where each $n_q \in \mathbb{N}_0$ and \[
		\sideset{}{'}\prod_{q \in \Sigma}(Q_q, Z_q) \cong \sideset{}{'} \prod_{q \in \Sigma} (\mathbb{Q}_q^{n_q}, \mathbb{Z}_q^{n_q}),
	\]
	Because of this, restricted products of this form will be denoted as $\prod_{q \in \Sigma}^{'}Q_q \cong \sideset{}{'}\prod_{q \in \Sigma}\mathbb{Q}_q^{n_q}$. \\
	As a corollary, we also have that $Q = \sideset{}{'}\prod_{q \in \Sigma}Q_q$ has no small $\mathbb{Z}\inv{S}$-submodules (because the identity neighbourhood \  $\prod_{q \in \Sigma}\mathbb{Z}_q^{n_q} \subset \sideset{}{'}\prod_{q \in \Sigma}\mathbb{Q}_q^{n_q}$ does not contain any nontrivial $\mathbb{Z}\inv{S}$-submodules).
\end{theorem}

\begin{proof}
	By Proposition 4-13 of \cite{RamakrishnanValenza1999}, for each $p \in \Sigma$, there exists an isomorphism $\phi_p: Q_p \to \mathbb{Q}_p^{n_p}$, for some $n_p \ge 0$, but we can suppose without loss of generality that $n_p \ge 1$ for convenience. \\
	As $Z_p$ is a compact open subgroup of $Q_p$, its image under $\phi_p$ is of the form  $\prod_{i = 1}^{n_p}p^{k_{i, p}}\mathbb{Z}_p$, with $k_{i, p} \in \mathbb{Z}$, by Corollary \ref{compact open subgroups of p-adic vector space}. But $\phi_p$ can be composed with the automorphism of $\mathbb{Q}_p^{n_p}$ that is multiplication by $(p^{-k_{i, p}})_{i = 1}^{n_p}$ and therefore we can suppose without loss of generality that $\phi_p(Z_p) = \mathbb{Z}_p^{n_p}$. \\
	So we can consider \begin{align*}
		f: \sideset{}{'}\prod_{q \in \Sigma} (Q_q, Z_q) & \to \sideset{}{'}\prod_{q \in \Sigma} (Q_q^{n_q}, \mathbb{Z}_q^{n_q}) \\
		(v_q)_{q \in \Sigma}                            & \mapsto (\phi_q(v_q))_{q \in \Sigma}.
	\end{align*}
	$f$ is a composition of homomorphisms, so it is an abelian group homomorphism (making it $\mathbb{Z}\inv{S}$-linear), and it is bijective as a consequence of each $\phi_p$ being an isomorphism that sends $Z_p$ to $\mathbb{Z}_p^{n_p}$. \\
	By Proposition \ref{guaranteed sigma compactness}, $\prod_{q \in \Sigma}^{'}(\mathbb{Q}_q^{n_q}, \mathbb{Z}_q^{n_q})$ is $\sigma$-compact, so as long as $f$ is open, which corresponds to saying that its inverse is continuous, $f$ will also be continuous and therefore an isomorphism by the Open Mapping Theorem. \\
	It suffices to check openness at the identity, so let \[
		U = \prod_{q \in \Sigma} U_q,
	\]
	where each $U_q$ is a compact open subgroup of $Q_q$, of the form  $\invimage{\phi_q}{\prod_{i=1}^{n_q} q^{l_{i, q}}\mathbb{Z}_q}$, for $l_{i, q} \in \mathbb{Z}$, with all but a finite number of the $U_q$'s different from $Z_q$. \\
	We have that \[
		f(U) = \prod_{q \in \Sigma} (\prod_{i = 1}^{n_q}  q^{l_{i, q}}\mathbb{Z}_q).
	\]
	This image is a subgroup of $\sideset{}{'}\prod_{q \in \Sigma}(\mathbb{Q}_q^{n_q}, \mathbb{Z}_q^{n_q})$ that contains $\prod_{q \in \Sigma}q^{l_q}\mathbb{Z}_q^{n_q}$, where $l_q = \max_{i = 1, \dots, n_q}\{l_{i, q}\}$. Because all but finitely many of the $U_q$'s are different from $Z_q$, it follows that all but finitely many of the $l_q$'s are non-zero, therefore $(q^{-l_q})_{q \in \Sigma}$ is an element of $\mathbb{A}_{\mathbb{Z}\inv{S}, f}$. \\
	As such, $\prod_{q \in \Sigma}q^{l_q}\mathbb{Z}_q^{n_q}$ is homeomorphic to $\prod_{q \in \Sigma}\mathbb{Z}_q^{n_q}$, through multiplication by $(q^{-l_q})_{q \in \Sigma}$ (since it is an invertible element in $\mathbb{A}_{\mathbb{Z}\inv{S}, f}$), therefore open.  \\
	In other words, $f(U)$ is a subgroup with non-empty interior, which makes it open. As $f$ sends basic open sets at the identity to open sets, $f$ is open.
\end{proof}

\begin{corollary}
	\label{self dual modules over finite adele ring}
	For $(Q_q, Z_q)_{q \in \Sigma}$ as above and $Q = \sideset{}{'}\prod_{q \in \Sigma}Q_q$, $Q$ is self-dual as a topological $\mathbb{Z}\inv{S}$-module.
\end{corollary}

\begin{proof}
	Since additive functions between $\mathbb{Z}\inv{S}$-modules are linear, it suffices to prove that $Q \cong \dual{Q}$ as topological abelian groups. By Theorem \ref{properties of restricted product}, \[
		\dual{Q} \cong \sideset{}{'}\prod_{q \in \Sigma} (\dual{Q_q}, A(Z_q)).
	\]
	Since each $A(Z_q)$ is a compact open subgroup of $\dual{Q_q}$ and each $\dual{Q_q}$ is a topological $\mathbb{Q}_q$-vector space isomorphic to $Q_q \cong \mathbb{Q}_q^{n_q}$ (this isomorphism exists because the $p$-adic numbers are self-dual), it follows from Theorem \ref{invariance of choice of compact open subgroup for p-adic vector spaces} that $Q \cong \dual{Q}$.
\end{proof}

\begin{theorem}
	\label{modules over finite adele ring}
	Let $Q \in LC_{\mathbb{A}_{\mathbb{Z}\inv{S}}, f}M$. Then $Q \cong \sideset{}{'}\prod_{q \in \Sigma}Q_q$, where $Q_q$ is a topological $\mathbb{Q}_q$-vector space for each $q \in \Sigma$. \\
\end{theorem}

\begin{proof}
	Let $C$ be a compact open subgroup of $Q$ (in particular it will be a $\mathbb{Z}_{\Sigma}$-submodule). By Proposition \ref{guaranteed sigma compactness}, \ we have $Q = \bigcup_{s \in S} \frac{1}{s}C$. For each $p \in \Sigma$, let $Q_p = Ann_Q(\prod_{q \in \Sigma\setminus p}^{'} \mathbb{Q}_q)$. Then $Q_p$ is a locally compact topological $\mathbb{Q}_p$-vector space by Corollary \ref{p-adic numbers are quotient of finite Adele ring} and Lemma \ref{modules over quotients}, therefore isomorphic to $\mathbb{Q}_p^{n_p}$ for some $n_p \in \mathbb{N}_0$, by Proposition 4-13 of \cite{RamakrishnanValenza1999}. \\
	At the same time, $Z_p := Q_p \cap C$ is a compact open subgroup of $Q_p$ and in particular a topological $\mathbb{Z}_p$-submodule of $Q_p$.\\
	For $p \in \Sigma$, let $e_p := (\delta_{q, p})_{q \in \Sigma}$ denote the unit vector of $\mathbb{Z}_{\Sigma}$ associated to $p$. Then given $m \in Q$, $e_pm \in Q_p$ and if $m \in C$, $e_pm \in Z_p$. \\
	We can first consider the map \begin{align*}
		f\colon C & \to \prod_{q \in \Sigma} Z_q   \\
		c         & \mapsto (e_qc)_{q \in \Sigma}.
	\end{align*}
	$f$ is a continuous group homomorphism and also injective: If $e_qc = 0$ for all $q \in \Sigma$, then $c = 1_{\mathbb{Z}_{\Sigma}}c = (\sum_{p \in \Sigma} e_q) c = \sum_{q \in \Sigma} e_qc = 0$. \\
	As for surjectivity, note that the image of $f$ is closed in $\prod_{q \in \Sigma} Z_q$ by compactness of $C$, therefore if it contains $\bigoplus_{q \in \Sigma} Z_q$, it has to be equal to the whole space. Given $(v_q)_{q \in \Sigma} \in \bigoplus_{q \in \Sigma} Z_q$, with $v_{q_1}, \dots, v_{q_k} \neq 0$, just consider $c = \sum_{i=0}^{k} v_{q_i}$, and we have that for $p \in \Sigma$, $e_pv_{q_i}$ is $0$ if $p \neq q_i$ (by definition) and $v_{q_i}$ otherwise (since in this case, $e_p$ acts like the identity on $Z_p$, due to its unique $\mathbb{Z}_p$-module structure). Therefore, $f(c) = (v_q)_{q \in \Sigma}$ and $f$ is surjective. Because $C$ is compact, $f$ is an isomorphism of topological groups. \\
	Since $Q = \bigcup_{s \in S} \frac{1}{s} C$, we can lift $f$ to the $\mathbb{Z}\inv{S}$-homomorphism \begin{align*}
		\overline{f}\colon Q & \to \sideset{}{'}\prod_{q \in \Sigma} (Q_q, Z_q) \\
		m = \frac{1}{s}c     & \mapsto \frac{1}{s}f(c).
	\end{align*}

	This is well defined because because for $s \in S$ and $p \in \Sigma$ not dividing $s$, $\frac{1}{s}e_p c \in Z_p$ for all $c \in C$ (this matters since only a finite number of primes divide $s$, guaranteeing that $\frac{1}{s}f(c)$ is in the restricted product), and because the family $(\frac{1}{s}C)_{s \in S}$ is directed with respect to inclusion: If $s, t \in S$, then $\frac{1}{s}C, \frac{1}{t}C \subset \frac{1}{st}C$.  \\

	$\overline{f}$ is an extension of a continuous group homomorphism, therefore continuous,  and is injective because of how it's defined. Surjectivity of $\overline{f}$ follows from surjectivity of $f$ and from the fact that \[
		\sideset{}{'} \prod_{q \in \Sigma} Q_q = \bigcup_{s \in S} \frac{1}{s}(\prod_{q \in \Sigma} Z_q).
	\]
	($\prod_{q \in \Sigma} Z_q$ is a compact open subgroup of $\sideset{}{'}\prod_{q \in \Sigma}(Q_q, Z_q)$ by Proposition \ref{properties of restricted product} and we can apply Proposition \ref{guaranteed sigma compactness}). \\
	So $\overline{f}$ is a continuous bijective $\mathbb{Z}\inv{S}$-homomorphism and since we know that $Q$ is $\sigma$-compact, it follows by the open mapping Theorem that $\overline{f}$ is an isomorphism.
\end{proof}

\begin{corollary}
	\label{modules over Adele ring}
	Let $L \in LC_{\mathbb{A}_{\mathbb{Z}\inv{S}}}M$. Then $L \cong \mathbb{R}^n \times \sideset{}{'}\prod_{q \in \Sigma}\mathbb{Q}_q^{n_q}$ as a topological $\mathbb{Z}\inv{S}$-module, for $n \ge 0$ and $n_q \ge 0$ for all $q \in \Sigma$. In particular, $L$ is self-dual by Corollary \ref{self dual modules over finite adele ring}, compactly generated (over $\mathbb{Z}\inv{S}$) and has no small submodules.
\end{corollary}

\begin{proof}
	We can consider $L' = Ann_L(0\times \mathbb{A}_{\mathbb{Z}\inv{S}, f})$ and $L'' = Ann_L(\mathbb{R} \times 0)$. Both of these are closed $\mathbb{Z}\inv{S}$-submodules of $M$, and admit topological $\mathbb{R}$-vector space and $\mathbb{A}_{\mathbb{Z}\inv{S}, f}$-module structures respectively, which makes them $\sigma$-compact in particular. \\
	Clearly, $L = L' + L''$ and $L' \cap L'' = 0$, so by Corollary \ref{sigma compact internal direct sum}, $L \cong L' \times L''$ and we can apply the two above theorems to get the isomorphism. \\
	L is compactly generated and has no small submodules because it is a product of two $\mathbb{Z}\inv{S}$ topological modules with both of those properties.
\end{proof}

\noindent  Because of the above Corollary and Theorem, $\mathbb{A}_{\mathbb{Z}\inv{S}, f}$ and $\mathbb{A}_{\mathbb{Z}\inv{S}}$-modules will be referred to without using any restricted product notation from now on, and they can be treated in the same way that finite dimensional real vector spaces are treated (since each such module is uniquely characterized by a sequence of dimensions).

\subsection{Structure Theorems}
\label{modules over localizations of integers}
It is first important to note that for any localization of $\mathbb{Z}$, we can find a set of primes uniquely associated to it that allows us to do a similar analysis to that of the previous subsection.\\
Let $S \subset \mathbb{Z}$ be multiplicatively closed. Then there exists $Z \subset \mathbb{Z}$ also multiplicatively closed such that $Z$ is multiplicatively generated by a set of prime numbers and $\mathbb{Z}\inv{S} = \mathbb{Z}\inv{Z}$: \\
Just take $Z$ to be the submonoid of $\mathbb{N}$ generated by all prime factors of elements of $S$. Since $S \subset Z$, clearly $\mathbb{Z}\inv{S} \subset \mathbb{Z}\inv{Z}$. \\
Conversely, if $\frac{1}{z} \in \mathbb{Z}\inv{Z}$, with $z = p_1^{\alpha_1} \dots p_n^{\alpha_n}$ and each $\alpha_i \ge 1$, for each $i \in \{1, \dots, n\}$ take $s_i \in S$ such that $s_i = p_i a_i$ for some $a_i \in \mathbb{Z}$. Then \[
	\frac{1}{z} = \frac{1}{p_1^{\alpha_1}\dots p^{\alpha_n}} = \frac{a_1^{\alpha_1} \dots a_n^{\alpha^n}}{s_1^{\alpha_1} \dots s_n^{\alpha_n}} \in \mathbb{Z}\inv{S}.
\]
As $1\inv{Z} \subset \mathbb{Z}\inv{S}$, it follows that $\mathbb{Z}\inv{Z} \subset \mathbb{Z}\inv{S}$. \\
This means that we can always suppose without loss of generality that $S$ is generated by a unique  set of prime numbers $\Sigma$, and we will do so. \\

We can now start the process of proving the Structure Theorems in $LC_{\mathbb{Z}\inv{S}}
	M$. Unlike in the classical case for $LCA$ groups, where the Theorems are proven in order, we will start by proving the 2nd Structure Theorem, which means characterizing the compactly generated modules.
\\
The lemma below is very simple, but it is one of the main factors that facilitates their study:

\begin{lemma}
	\label{real vector space is submodule}
	Let $M \in LC_{\mathbb{Z}\inv{S}}M$ with $(M, +) = \mathbb{R}^n \times H$ as given by the 1st Structure Theorem for $LCA$ groups. Then $\mathbb{R}^n \times 0$ is a $\mathbb{Z}\inv{S}$-submodule of $M$. In particular, $H$ is also a $\mathbb{Z}\inv{S}$-submodule
\end{lemma}

\begin{proof}
	Given $(x, 0) \in \mathbb{R}^n \times 0$ and $s \in S$, we can note that $s(\frac{1}{s}x, 0) = (x, 0) = (s\frac{1}{s})(x, 0)$. The fact that $M$ is a $\mathbb{Z}\inv{S}$-module implies therefore that $\frac{1}{s}(x, 0) = (\frac{1}{s}x, 0)$ which means that $\mathbb{R}^n \times 0$ is a closed submodule of $M$. \\
	This makes $H$ a module over $\mathcal{Z}\inv{S}$-module, by being a quotient of $M$. Therefore in particular, it is $S$-divisible, that is, $sH = H$ for all $s \in S$. So a product of the form $\frac{1}{s}(0, h)$ will be an element of $H$ since $(0, h)$ is of the form $s(0, g)$ for some $g \in H$ and because of how $\mathcal{Z}\inv{S}$-module structures work.
\end{proof}

\noindent  Because of this, we can suppose without loss of generality that $M$ contains a compact open subgroup.

\begin{prop}
	\label{compactly generated splits}
	Let $M \in LC_{Z\inv{S}}M$ as above be compactly generated. Then $M \cong (\mathbb{Z}\inv{S})^k \times C(M)$ for some $k \ge 0$.
\end{prop}

\begin{proof}
	Since $M$ contains a compact open subgroup, $M/C(M)$ is discrete and compactly generated, which makes it finitely generated. At the same time, it contains no nontrivial finite subgroups by Corollary \ref{elliptic and t.d. duality}. In particular, it is torsion free over $\mathbb{Z}$, which is equivalent to being torsion free over $\mathbb{Z}\inv{S}$. Therefore, by the Structure Theorem for finitely generated modules over a PID, $M/C(M)$ is free of finite dimension, which means we get a continuous surjective map $M \to (\mathbb{Z}\inv{S})^k$ for some $k \ge 0$, with kernel $C(M)$. By Corollary \ref{free quotient implies split}, $M \cong (\mathbb{Z}\inv{S})^k\times C(M)$.
\end{proof}

The above lemma and proposition, combined, tell us that it suffices to classify elliptic compactly generated modules over $\mathbb{Z}\inv{S}$ in order to arrive at the 2nd Structure Theorem.
However, it is useful to first look at the totally disconnected ones.

\begin{lemma}
	\label{compactly generated elliptic and totally disconnected module}
	Let $N \in LC_{\mathbb{Z}\inv{S}}M$ be compactly generated, elliptic and totally disconnected. Then $N \cong Q \times C$, where $Q$ is a topological module over $\mathbb{A}_{\mathbb{Z}\inv{S}, f}$, and $C$ is a compact $\mathbb{Z}_{\mathbb{P}\setminus \Sigma}$-topological module.
\end{lemma}

\begin{proof}
	$N$ is a topological module over $\hat{\mathbb{Z}} \cong \prod_{q \in \mathbb{P}} \mathbb{Z}_q \cong \mathbb{Z}_\Sigma \times \prod_{q \in \mathbb{P} \setminus \Sigma} \mathbb{Z}_q = \mathbb{Z}_{\Sigma} \times \mathbb{Z}_{\mathbb{P}\setminus \Sigma}$, by Theorem \ref{characterization of modules over profinite integers}. With this identification of the profinite integers, we write its elements in the form $(x, y)$ where $x \in \mathbb{Z}_{\Sigma}$ and $y \in \mathbb{Z}_{\mathbb{P}\setminus \Sigma}$. \\
	Let $N' = Ann_N(0 \times \mathbb{Z}_{\mathbb{P}\setminus \Sigma})$ and $N'' = Ann_N(\mathbb{Z}_{\Sigma} \times 0)$. \\
	$N'$ and $N''$ are closed subgroups of $N$, and can also be easily seen to be $\mathbb{Z}\inv{S}$-submodules.
	\\
	Given $n \in N$, we have $n = (1, 1)n = (1, 0)n + (0, 1)n \in N'' + N'$, therefore $N = N' + N''$. It is also clear that $N' \cap N'' = 0$. Since $N$ is compactly generated over $\mathbb{Z}\inv{S}$, it is also $\sigma$-compact, which implies that $N'$ and $N''$ have the same property. As such, $N \cong N' \times N''$ by Corollary \ref{sigma compact internal direct sum}. \\
	$N'$ is a topological module over $\frac{\hat{\mathbb{Z}}}{\mathbb{Z}_{\mathbb{P}\setminus \Sigma}} \cong \mathbb{Z}_\Sigma$ and over $\mathbb{Z}\inv{S}$ at the same time. By Proposition \ref{canonical localization relation}, $N'$ admits an obvious module structure over $\mathbb{A}_{\mathbb{Z}\inv{S}, f}$. The action given by this structure is also continuous, because $\bigcup_{s \in S}\frac{1}{s}\mathbb{Z}_{\Sigma} \times N'$ is an open cover of $\mathbb{A}_{\mathbb{Z}\inv{S}, f} \times N'$, and the restriction of the action to each $\frac{1}{s}\mathbb{Z}_{\Sigma} \times N'$ \ is the composition of the continuous $\mathbb{Z}_{\Sigma}$-module structure  of $N'$ with multiplication by $\frac{1}{s} \in \mathbb{Z}\inv{S}$. Therefore $N'$ is a topological module over $\mathbb{A}_{\mathbb{Z}\inv{S}, f}$ and we can take $Q = N'$. \\

	As for $N''$, note that since it is a quotient of $N$, it is also compactly generated over $\mathbb{Z}\inv{S}$, while being a topological module over $\mathbb{Z}_{\mathbb{P}\setminus \Sigma}$ at the same time. In particular, we can write $N'' = \bigcup_{s \in S} \frac{1}{s}K$, where $K$ is some compact open subgroup of $N''$, by Corollary \ref{compact subset contained in compact open subgroup}. But given $s \in S$, $s$ is not divisible by any prime $q \in \mathbb{P}\setminus \Sigma$, which makes it invertible in each associated $\mathbb{Z}_q$. \\
	As such, given $k \in K$, \[
		\frac{1}{s}k = \frac{1}{s}((1_{\mathbb{Z}_q})_{q \in \mathbb{P}\setminus\Sigma} \ k) = (\frac{1_{\mathbb{Z}_q}}{s})_{q \in \mathbb{P}\setminus \Sigma} k \in \mathbb{Z}_{\mathbb{P}\setminus \Sigma}K.
	\]
	As $K$ is a closed subgroup of $N''$, it is also a $\mathbb{Z}_{\mathbb{P}\setminus \Sigma}$-submodule, therefore $\mathbb{Z}_{\mathbb{P}\setminus \Sigma}K \subset K$ and $\frac{1}{s}K \subset K$ for every $s \in S$. In other words, $N'' = K$ is compact and we can take $C = N''$.
\end{proof}

\begin{corollary}
	\label{elliptic and compactly generated}
	Let $M \in LC_{\mathbb{Z}\inv{S}}M$ be compactly generated and elliptic. Then $M \cong Q \times K$, where $Q$ is a topological $\mathbb{A}_{\mathbb{Z}\inv{S}, f}$-module and $K$ is a compact $\mathbb{Z}\inv{S}$-module  such that $K/K_0$ is a topological $\mathbb{Z}_{\mathbb{P}\setminus \Sigma}$-module and $K_0 \cong M_0$. For convenience purposes, we will simply say that $K$ is a compact $\mathbb{Z}_{\mathbb{P}\setminus \Sigma}$ extension by $M_0$ (since there exists a short exact sequence $0 \longrightarrow M_0 \longrightarrow K \longrightarrow K/K_0 \longrightarrow 0$).
\end{corollary}

\begin{proof}
	Since $M/M_0$ is in the conditions of the above Lemma, $M/M_0 \cong T \times C$, where $T$ is a topological $\mathbb{A}_{\mathbb{Z}\inv{S}, f}$-module and $C$ is a compact topological $\mathbb{Z}_{\mathbb{P}\setminus \Sigma}$-module.  Therefore,
	\[
		C(\dual{M}) = A(M_0) \cong \dual{M/M_0} \cong \dual{T} \times B,
	\]
	where $B = \dual{C}$ is a discrete topological $\mathbb{Z}_{\mathbb{P}\setminus \Sigma}$-module and therefore in particular, torsion. We can suppose that $\dual{T}$ and $B$ are contained in $C(\dual{M})$, with $\dual{T} + B = C(\dual{M})$ and $\dual{T}\cap B = 0$, for convenience.
	Since $\dual{T}$ is in particular a vector space over $\mathbb{Q}$, it is injective as a $\mathbb{Z}\inv{S}$-module. \\
	At the same time, because of $B$ being discrete, $\dual{T}$ is an open injective submodule of $C(\dual{M})$ which is also open in $\dual{M}$ (since $\dual{M}$ contains a compact open subgroup). As such, by Corollary \ref{Open injective submodule}, there exists a discrete submodule $D$ of $\dual{M}$ such that $\dual{M} \cong \dual{T} \times D$ (in particular, $D \cap \dual{T} = 0$ as well). Because $B$ is a torsion abelian group and $\dual{T}$ is torsion-free (reminder that $T \cong \dual{T}$), $B$ has to be contained in $D$ and indeed, we can note that \[
		C(D) = D \cap C(\dual{M}) = D \cap (\dual{T} + B) = D \cap B = B,
	\]
	where the third equality holds because $B \subset D$ and $D \cap \dual{T} = 0$. \\
	Dualizing, we get that \[
		M \cong T \times K,
	\]
	where $K = \dual{D}$ is compact and $K_0 = A_D(C(D)) = A(B)$. \\
	$K/K_0 = \dual{D}/A_D(C(D)) \cong \dual{C(D)} = \dual{B} \cong C$ is a topological $\mathbb{Z}_{\mathbb{P}\setminus \Sigma}$ module. \\
	$K_0 \cong M_0$ because under the above isomorphism, $M_0 \cong 0 \times K_0$.
\end{proof}

\noindent With this, we can arrive at the 2nd Structure Theorem for $LC$ modules over $\mathbb{Z}\inv{S}$.

\begin{theorem}[2nd Structure Theorem]
	\label{2nd Structure Theorem}
	Let $M \in LC_{\mathbb{Z}\inv{S}}M$ be compactly generated. Then putting Lemma \ref{real vector space is submodule} together with Proposition \ref{compactly generated splits} and Corollary \ref{elliptic and compactly generated},\[
		M \cong A \times (\mathbb{Z}\inv{S})^k \times K,
	\]
	for $k \ge 0$, $A$ a topological module over $\mathbb{A}_{\mathbb{Z}\inv{S}}$ and $K$ a compact $\mathbb{Z}_{\mathbb{P}\setminus \Sigma}$ extension by $C(M)_0 = C(M)\cap M_0 \cong K_0$.
\end{theorem}

\begin{proof}
	By Lemma \ref{real vector space is submodule} and Proposition \ref{compactly generated splits}, $M \cong \mathbb{R}^n \times (\mathbb{Z}\inv{S})^k \times C(M)$. Applying Corollary \ref{elliptic and compactly generated}, we get the intended result.
\end{proof}

The above Structure Theorem is sufficient to prove the 1st one, because locally compact modules contain atleast one open compactly-generated submodule.
\begin{theorem}[1st Structure Theorem]
	\label{1st Structure Theorem}
	Let $M \in LC_{\mathbb{Z}\inv{S}}M$. Then \[
		M \cong A \times N,
	\]
	where $A$ is a topological $\mathbb{A}_{\mathbb{Z}\inv{S}}$-module, and $N$ contains a compact open $\mathbb{Z}_{\mathbb{P}\setminus \Sigma}$ extension by $N_0 \cong M_0 \cap C(M)$ (such extension is always taken to be a submodule).
\end{theorem}

\begin{proof}
	Let $V$ be a compact neighbourhood of the identity in $M$ and consider $M'$ to be the submodule generated by $V$. $M'$ is an open compactly generated submodule of $M$ and by the 2st Structure Theorem, there exists an isomorphism $\phi\colon A \times (\mathbb{Z}\inv{S})^k \times K \to M'$ as above, where $K$ is a $\mathbb{Z}_{\mathbb{P}\setminus \Sigma}$ extension by $(M')_0 \cap C(M')$, the latter identified with $0\times 0 \times K_0$ through $\phi$. \\
	Let $Q = \phi(A \times 0 \times 0)$ and $L = \phi(0 \times (\mathbb{Z}\inv{S})^k \times K)$. Then $L + Q = M'$, which is open in $M$, $L$ and $Q$ are both $\sigma$-compact with $L \cap Q = 0$ and $Q$ is an injective $\mathbb{Z}\inv{S}$-module. As such, by Lemma \ref{lemma 4.2.14}, there exists a closed submodule $N$ of $M$ such that $L$ is open in $N$, and $M \cong Q \times N \cong A \times N$. \\
	$K' = \phi(0 \times 0 \times K)$ is open in $L$ therefore also open in $N$. At the same time, it is a compact $\mathbb{Z}_{\mathbb{P}\setminus \Sigma}$-extension by $\phi(0 \times 0 \times K_0) = (K')_0 = N_0$ (Corollary \ref{Connected component of open subgroup}). Under the above isomorphism on $M$, $M_0 \cap C(M)$ is identified with $N_0$, proving the result.

\end{proof}

Before getting to the third Theorem for modules, it is important to look at the 3rd Structure Theorem of $LCA$ groups in a different perspective: It says that $G \in LCA$ is of Lie type if and only if it is isomorphic to $\mathbb{R}^n \times (\unitcircle)^k \times D$, where $n, k \ge 0$ and $D$ is discrete. \\
But we can note that if $H$ is an abelian Lie group (of finite dimension), it will also be of this form: $H_0$ is a connected abelian Lie group of finite dimension, therefore of the form $\mathbb{R}^n \times (\unitcircle)^k$, and it is also open in $H$ since $H$ is locally connected. So $H$ contains an open injective subgroup, which implies $H \cong \mathbb{R}^n \times (\unitcircle)^k \times D$, for $D$ discrete, by Corollary \ref{Open injective submodule}. \\
In other words, the 3rd Structure Theorem implies that $G \in LCA$ is of Lie type if and only if it is isomorphic (as a topological group) to an abelian Lie group, which is also equivalent to admitting a Lie group structure. At the same time, Hidehiko Yamabe proved in \cite{Yamabe1953} that a locally compact group with the no small subgroups property is a Lie group (in the sense of admitting a Lie group structure), and it is also known that any Lie group of finite dimension has such property. \\
This means that the 3rd Structure Theorem for $LCA$ groups can be slightly altered to say that $G \in LCA$ has the NSS property if and only if $G \cong \mathbb{R}^n \times (\unitcircle)^k \times D$, for $n, k$ and $D$ as above. \\
With this in mind, in the case of modules we deal with the no small submodules property as well.

\begin{theorem}[3rd Structure Theorem]
	\label{3rd Structure Theorem}
	Let $M \in LC_{\mathbb{Z}\inv{S}}M$ have the no small submodules property. Then \[
		M \cong A \times (\dual{\mathbb{Z}\inv{S}})^k \times D,
	\]
	where $k \ge 0$, $A$ is a topological $\mathbb{A}_{\mathbb{Z}\inv{S}}$-module and $D$ is discrete, with $C(D) = T(D)$ a topological $\mathbb{Z}_{\mathbb{P}\setminus \Sigma}$ module.
\end{theorem}

\begin{proof}
	By Proposition \ref{no small submodules implies dual is compactly generated} and Pontryagin Duality, $\dual{M}$ is compactly generated, therefore isomorphic to $A \times (\mathbb{Z}\inv{S})^k \times K$, where $n, k \ge 0$,  $A$ is a topological $\mathbb{A}_{\mathbb{Z}\inv{S}}$-module, and $K$ is a compact $\mathbb{Z}_{\mathbb{P}\setminus \Sigma}$ extension by $(\dual{M})_0 \cap C(\dual{M}) \cong 0 \times 0 \times K_0$. \\
	By dualizing, \[
		M \cong A \times (\dual{\mathbb{Z}\inv{S}})^k \times D,
	\]
	where $D = \dual{K}$ contains $C(D) = A_K(K_0) \cong \dual{K/K_0}$, which is a topological $\mathbb{Z}_{\mathbb{P}\setminus \Sigma}$-module since $K/K_0$ is. $A \cong \dual{A}$ by Corollary \ref{modules over Adele ring}.
\end{proof}

\begin{corollary}
	\label{duality betweeen compactly generated and nssm}
	Let $M \in LC_{\mathbb{Z}\inv{S}}M$. Then $M$ is compactly generated if and only if $\dual{M}$ has no small submodules.
\end{corollary}

\begin{proof}
	One of the implications was already given in Proposition \ref{no small submodules implies dual is compactly generated}. If $M$ is compactly generated, then we can apply the 2nd Structure Theorem, along with Corollary \ref{modules over Adele ring} and Proposition \ref{fin generated implies no small submodules on dual} to conclude that $\dual{M}$ has no small submodules.
\end{proof}

This means that even though $LC$ modules over $\mathbb{Z}\inv{S}$ satisfying the no small submodules property are not equivalent to $\mathbb{Z}\inv{S}$-modules of Lie type given in Definition \ref{Definition of module of Lie type}, they satisfy the analogous property to the $LCA$ groups of Lie type, which is precisely being the Pontryagin Duals of compactly generated modules in $LC_{\mathbb{Z}\inv{S}}M$. \\

\begin{corollary}
	Let $M \in LC_{\mathbb{Z}\inv{S}}M$. \begin{itemize}
		\item $M$ is an inverse limit (in $LC_{\mathbb{Z}\inv{S}}M$) of $LC$ $\mathbb{Z}\inv{S}$-modules that have no small submodules .
		\item Any neighbourhood of the identity in $M$ contains a compact submodule $K$ such that $M/K$ contains no small submodules.
	\end{itemize}
	The first statement is a direct consequence of the equivalence of categories $LC_{\mathbb{Z}\inv{S}}M \cong (LC_{\mathbb{Z}\inv{S}}M)^{op}$ given by Pontryagin Duality, along with the fact that $M \in LC_{\mathbb{Z}\inv{S}}M$ is the direct limit of its open compactly generated submodules, and Corollary \ref{duality betweeen compactly generated and nssm}. \\
	The second statement is a non-categorical analogue of the first, and its proof is given below.
\end{corollary}

\begin{proof}
	Let $U$ be an open neighbourhood of $0$ in $M$. We can suppose without loss of generality that \[
		U = \{m \in M: \ \abs{\chi(m) - 1} < \epsilon \ \forall \chi \in V\},
	\]
	where $V$ is a compact neighbourhood of the identity in $\dual{M}$ and $\epsilon > 0$, by Corollary \ref{immediate consequences}. \\
	Consider $K = \orth{(\mathbb{Z}\inv{S}V)}$. Then $K$ is a submodule of $M$ contained in $U$, and $K$ is compact since $V$ has non-empty interior. At the same time, \[
		\dual{M/K} \cong A(K) = A(\orth{(\mathbb{Z}\inv{S}V)}) = \mathbb{Z}\inv{S}V \ (\mathbb{Z}\inv{S}V\ \text{is open in} \ \dual{M}),
	\]
	which means that $\dual{M/K}$ is compactly generated. By Corollary \ref{duality betweeen compactly generated and nssm}, $M/K$ has no small submodules.
\end{proof}

\begin{corollary}[Classification of locally compact vector spaces over $\mathbb{Q}$]
	\label{LC vector spaces over the rationals}
	Let $V \in LC_{\mathbb{Q}}M$. Then $V \cong A \times  \mathbb{Q}^{(I)} \times \dual{\mathbb{Q}}^J$, where $n \ge 0$, $A \in LC_{\mathbb{A}_{\mathbb{Q}}}M$ and $I, J$ are arbitrary sets.
\end{corollary}

\begin{proof}
	We know that $V \cong A \times N$ as given by the 1st Structure Theorem. However, in this case, $N$ contains a compact open $\mathbb{Q}$-vector space K, which makes a big difference as any vector space over $\mathbb{Q}$ is injective. By Lemma \ref{Open injective submodule}, there exists a discrete $\mathbb{Q}$-subvectorspace D of $N$ such that $N \cong K \times D$. By Pontryagin Duality, it is easy to classify all compact $\mathbb{Q}$-vector spaces, (reminder that Pontryagin Duality swaps direct sums of discrete modules with products of compact modules) and we get that $K \cong \dual{\mathbb{Q}}^J$, $D \cong \mathbb{Q}^{(I)}$.
\end{proof}

\subsection{Modules of Lie type}
\label{Modules of Lie type}

We will take for granted here that $\dual{\mathbb{Z}\inv{S}}$ is isomorphic to $\frac{\mathbb{A}_{\mathbb{Z}\inv{S}}}{\mathbb{Z}\inv{S}}$, where $\mathbb{Z}\inv{S}$ embeds anti-diagonally through $x \mapsto (x, -x, -x, -x, \dots)$. A proof of this fact is given for the case $\mathbb{Z}\inv{S} = \mathbb{Q}$ in \cite{ConradCharacterQ}, and for the generalized case the proof is practically the same, but instead of working over all prime numbers, we work with prime numbers in $\Sigma$ (note that in the case where $\Sigma$ is empty, we get the quotient $\mathbb{R}/\mathbb{Z} \cong \unitcircle$, as expected). \\
The anti-diagonal embedding of $\mathbb{Z}\inv{S}$ has a discrete image, since its intersection with $]-1, 1[ \times \mathbb{Z}_{\Sigma}$ is trivial. This means that $\dual{\mathbb{Z}\inv{S}}$ is locally isomorphic to $\mathbb{A}_{\mathbb{Z}\inv{S}}$ and locally compact $\mathbb{Z}\inv{S}$ modules of Lie type are those locally isomorphic to finite powers of $\mathbb{A}_{\mathbb{Z}\inv{S}}$. \\

Let $M \in LC_{\mathbb{Z}\inv{S}}M$ be of Lie type and $\phi\colon U \to V \subset M$ a local isomorphism from $\mathbb{R}^l \times (\mathbb{A}_{\mathbb{Z}\inv{S}, f})^l \cong (\mathbb{A}_{\mathbb{Z}\inv{S}})^l$ to $M$, for some $l \ge 0$. We can suppose $U = U_1 \times U_2$, where $U_1$ and $U_2$ are open identity neighbourhoods in $\mathbb{R}^l$ and $(\mathbb{A}_{\mathbb{Z}\inv{S}, f})^l$, respectively. \\
Analogously to the trick done in Lemma 4.2.19 of \cite{DeitmarEchterhoff2014}, we can extend $\phi$ to a continuous homomorphism $\psi\colon \mathbb{R}^l \times (\mathbb{A}_{\mathbb{Z}\inv{S}, f})^l$ as follows: Given $(v, a) \in \mathbb{R}^l \times (\mathbb{A}_{\mathbb{Z}\inv{S}, f})^l$, choose $s \in S$ such that $\frac{1}{s}v \in U_1$ and $sa \in U_2$, and define \[
	\psi(v, a) := s\phi(\frac{1}{s}v) + \frac{1}{s}\phi(sa).
\]
This is well defined: If $s, t \in S$ such that $\frac{1}{s}v, \frac{1}{t}v \in U_1$ and $sa, ta \in U_2$, then \[
	s\phi(\frac{1}{s}v) + \frac{1}{s}\phi(sa) = st\phi(\frac{1}{st}v) + \frac{1}{st}\phi(sta) = t\phi(\frac{1}{t}v) + \frac{1}{t}\phi(ta),
\]
where the above manipulations make sense precisely because $\phi$ is a local isomorphism of $\mathbb{Z}\inv{S}$-modules. \\

$\psi$ is additive: If $(v, a), (w, b) \in \mathbb{R}^l \times (\mathbb{A}_{\mathbb{Z}\inv{S}, f})^l$, choose $s \in S$ such that $\frac{1}{s}(v + w), \frac{1}{s}v, \frac{1}{s}w \in U_1$ and $s(a+b), sa, sb \in U_2$. Then \begin{align*}
	 & \psi(v+w, a+b) = s\phi(\frac{1}{s}(v+w)) + \frac{1}{s}\phi(s(a + b)) =                                   \\
	 & s(\phi(\frac{1}{s}v) + \phi(\frac{1}{s}w)) + \frac{1}{s}(\phi(sa) + \phi(sb)) = \psi(v, a) + \psi(w, b),
\end{align*}
where the above manipulations make sense because $\phi$ is locally additive. \\
$\psi$ being additive also makes it a $\mathbb{Z}\inv{S}$ homomorphism, and continuity at the identity makes $\psi$ continuous. Because $\psi|_{U} = \phi$, $\psi$ is locally injective and therefore has discrete kernel. \\
At the same time, $Im \ \psi$ contains $\phi(U)$, which is open and as such, $\psi$ has open image in $M$. As $\mathbb{R}^l \times (\mathbb{A}_{\mathbb{Z}\inv{S}, f})^l$ is divisible over $\mathbb{Z}\inv{S}$, $Im \ \psi$ also is, therefore being injective over $\mathbb{Z}\inv{S}$.\\
So $Im \ \psi$ is an open injective submodule of $M$ and by Corollary \ref{Open injective submodule}, $M \cong Im \ \psi \times D$, where $D$ is discrete. At the same time, because $\mathbb{R}^l \times (\mathbb{A}_{\mathbb{Z}\inv{S}, f})^l$ is $\sigma$-compact, $Im \ \psi \cong \frac{\mathbb{R}^l \times (\mathbb{A}_{\mathbb{Z}\inv{S}, f})^l}{\ker \psi}$, by the Open Mapping Theorem. \\
This means that classifying locally compact $\mathbb{Z}\inv{S}$-modules of Lie type in a similar way to what is given in the 3rd Structure Theorem for $LCA$ groups amounts to classifying quotients of finite powers of $\mathbb{A}_{\mathbb{Z}\inv{S}}$ by discrete submodules. \\
We also know that given $M \in LC_{\mathbb{Z}\inv{S}}M$ of Lie type, $M$ has no small submodules by Proposition \ref{Lie type implies NSSM}, which means we can see $M$ under the guise of the the 3rd Structure Theorem. \\
This leads to the following questions to end this article: \begin{enumerate}
	\item What are the discrete $\mathbb{Z}\inv{S}$-submodules of $(\mathbb{A}_{\mathbb{Z}\inv{S}})^l$ for $l \ge 1$, and what are their associated quotients up to isomorphism?
	\item Given $M \in LC_{\mathbb{Z}\inv{S}}M$ having no small submodules and \\ $M \cong \mathbb{R}^n \times \sideset{}{'}\prod_{q \in \Sigma}\mathbb{Q}_q^{n_q} \times (\dual{\mathbb{Z}\inv{S}})^k \times D$ as given by the the 3rd Structure Theorem, are there conditions on the numbers $n, (n_q)_{q \in \Sigma}$ and $k$ that characterize whether $M$ is of Lie type?
\end{enumerate}

\section*{Acknowledgments}
This research was conducted in partial fullfilment of the requirements for the Master's degree in Mathematics at the University of Porto, under the supervision of  Carlos André, from University of Lisbon, and co-supervision of  Samuel Lopes from the University of Porto, whose guidance and support are greatly appreciated.

\appendix{}
\section{Some elementary facts about topological groups and modules}

\begin{lemma}
	\label{conn.component}
	If $G$ is a topological group and $G_0$ is its connected component at the identity, $G_0$ is contained in any open subgroup of $G$. \\
	Furthermore, $G_0$ is a closed normal subgroup of $G$ such that the quotient $G/G_0$ is totally disconnected, and given $x \in G$, $xG_0$ is the connected component of $x$.
\end{lemma}

\begin{theorem}
	\label{Locally compact t.d.}
	If a locally compact topological group $G$ is totally disconnected, then it admits a neighbourhood basis at the identity consisting of compact open subgroups. The converse of this statement is a direct consequence of the previous lemma (since in such case, $G_0 = 0$). \\
	Furthermore, if the group is compact, the compact open subgroups can be taken to be normal.
\end{theorem}

\noindent{}  Proof of both above results is found, for example, at the beginning of chapter 4 in \cite{DeitmarEchterhoff2014}.

\begin{corollary}
	\label{Connected component of open subgroup}
	Let $G$ be a locally compact group. Then \[
		G_0 = \bigcap_{\mathclap{\substack{O \le G \\ O \ open}}} O.
	\] Furthermore, if $H$ is an open subgroup of $G$, $H_0 = G_0$.
\end{corollary}

\begin{proof}
	Let $\pi\colon G \to G/G_0$ be the quotient map and let $x \in G$ such that $x$ is in all open subgroups of $G$. Then $\pi(x)$ is in all open subgroups of $G/G_0$ (since those are all of the form $O/G_0$, where $O$ is an open subgroup of $G$) and therefore $\pi(x) = 1$ by Theorem \ref{Locally compact t.d.}, which means that $x \in G_0$. \\
	Let $H \le G$ be open. Then $H$ is also locally compact so by the first statement of this corollary, \[
		H_0 = \bigcap_{\mathclap{\substack{O \le H \\ O \ open}}} O = \bigcap_{\mathclap{\substack{O \le G \\ O \ open}}} O \cap H = G_0 \cap H = G_0,
	\]
	since $G_0 \subset H$.
\end{proof}

\begin{theorem}
	\label{quotient by compact subgroup is compact}
	Given $K$ a compact subgroup of $G$, $G/K$ is compact if and only if $G$ is compact.
\end{theorem}

\begin{theorem}
	\label{Guaranteed uniform continuity}
	Let $G$, $G'$ be topological abelian groups, $S \subset G$ and $f\colon S  \to G'$ a function.
	If $f$ is a continuous map and $S$ is compact, then $f$ is necessarily uniformly continuous.
\end{theorem}

\begin{proof}
	Let $U$ be a neighbourhood of the identity in $G'$ and $U'$ such that $U' + U' \subset U$ and $U'$ is symmetric. For  each $x \in S$, there exists a symmetric unit neighbourhood $W_x$ in $G$ such that $f(x+ W_x) \subset f(x) + U'$, by continuity of $f$. \\
	Let $V_x$ be an open symmetric neighbourhood of $0$ in $G$ such that $V_x + V_x \subset W_x$. $\{x + V_x: x \in S\}$ is an open cover of S, so by compactness, \[S \subset \bigcup_{i=1}^n x_i + V_i,\]
	with $V_i := V_{x_i}$, for a finite family $x_1, \dots, x_n \in G$.

	Define $V = \bigcap_{i=1}^n V_i$. V is a symmetric open neighbourhood of $0$ in $G$. Let $x, y \in S$ with $x = x_k + v_k$, $y = x_l + v_l$, for some $k, l \in \{1, \dots, n\}$, and $v_k \in V_k$, $v_l \in V_l$. \\
	Suppose $x - y \in V$. $x - x_l = x - (y - v_l) = (x-y) + v_l \in V + V_l \subset V_l + V_l \subset W_{x_l}$. Therefore, by definition of $W_{x_l}$, $f(x) = f(x_l + (x-x_l)) \subset f(x_l) + U'$. Analogously, $y - x_l = v_l \in V_l \subset W_{x_l}$, so $f(y) = f(x_l + (y-x_l)) \subset f(x_l) + U'$. As such, $f(x) - f(y) \in f(x_l) + U' - f(x_l) - U' = U' - U' = U' + U' \subset U$.
\end{proof}

\begin{prop}
	\label{product of sigma compact}
	If $G$ is a topological group and $A, B \subset G$ are $\sigma$-compact, then $AB$ also is.
\end{prop}

\begin{proof}
	Write $A = \bigcup_{n \in \mathbb{N}} A_n$ and $B = \bigcup_{k \in \mathbb{N}} B_k$, where each $A_n$ and $B_k$ are compact. \\
	Then by definition, $AB = \bigcup_{n = 1}A_n B = \bigcup_{n \in \mathbb{N}}A_n(\bigcup_{k \in \mathbb{N}}B_k) = \bigcup_{n \in \mathbb{N}}\bigcup_{k \in \mathbb{N}}A_nB_k$, which is a countable union of compact sets.
\end{proof}

\begin{prop}
	\label{compactly generated over sigma compact ring}
	Let $R$ be a $\sigma$-compact topological ring and $M$ a compactly generated topological $R$-module. Then $M$ is also $\sigma$-compact.
\end{prop}

\begin{proof}
	Write $R = \bigcup_{n = 1}^{+\infty} K_n$, where each $K_n$ is compact and $M = RC$, for some $C \subset M$ compact. \\
	Letting $M_k = \{\sum_{i=1}^k r_ic_i: \ r_i \in R, \ c_i \in C\}$ for each $k \in \mathbb{N}$ we get that by definition of submodule generated by a set, $M = \bigcup_{k = 1} M_k$, so it suffices to show that each $M_k$ is $\sigma$-compact.  \\
	Defining $\underline{XC} := \{xc: \ x \in X, \ c \in C\}$ for $X \subset R$, we get that $M_k = \underbrace{\underline{RC} + \dots + \underline{RC}}_{\text{k \ times}}$, and each $\underline{RC}$ is equal to $\bigcup_{n = 1}^{+ \infty}\underline{K_nC}$, which is $\sigma$-compact by continuity of multiplication. \\
	Since, each $M_k$ is a finite sum of $\sigma$-compact sets, $M$ is $\sigma$-compact by Proposition \ref{product of sigma compact}.
\end{proof}

It is well known, algebraically, that given a ring $R$ and an ideal $I$, there is a correspondence between $R/I$-modules and $R$-modules annihilated by $I$. The following lemma shows that this correspondence also holds when talking about topological modules.

\begin{lemma}
	\label{modules over quotients}
	Let $R$ be a topological ring and $M$ a topological $R$-module. If $I$ is an ideal of $R$ such that $IM = 0$, then $M$ admits a topological $R/I$-module structure. Conversely, any topological $R/I$-module structure on $M$ induces a topological $R$-module structure on $M$ such that $IM = 0$.
\end{lemma}

\begin{proof}
	It is well known that the correspondence given in the lemma works on an algebraic level, so we only need to check that the respective maps are continuous. \\
	Let $f: R \times M \to M$ be the usual module structure and \begin{align*}
		f'\colon R/I \times M & \to M      \\
		(r+I, m)              & \mapsto rm
	\end{align*}
	the induced action. If $\pi: R\times M \to R/I \times M$ is the quotient map, then given $U \subset M$ open, \[
		\invimage{f'}{U} = \{(r + I, m): \ rm \in U\} = \pi(\invimage{f}{U}),
	\]
	which is open since $\pi$ is an open map and $f$ is continuous. \\
	If $M$ is a topological $R/I$-module, then \begin{align*}
		f\colon R \times M & \to M            \\
		(r, m)             & \mapsto (r + I)m
	\end{align*}
	is continuous since it is equal to $f' \circ \pi$ and $f'$ is continuous (the topology on $R/I \times M$ is naturally identified with the quotient topology on $\frac{R \times M}{I \times 0}$, which makes this argument work).
\end{proof}

\subsection{profinite integers and finite dimensional p-adic vector spaces}

\begin{definition}
	\label{definition of profinite integers}
	The ring of profinite integers $\hat{\mathbb{Z}}$ is the set of formal series of the form \[
		\hat{\mathbb{Z}} \ni x = \sum_{i=0}^{+ \infty} c_i i!,
	\]
	where $0 \le c_i \le i$ for all $i \ge 0$. The topology on $\hat{\mathbb{Z}}$ is induced by the neighbourhood basis at the identity consisting of the sets of the form $n\hat{\mathbb{Z}}$ for $n \in \mathbb{N}$. \\
	The integers embed into $\hat{\mathbb{Z}}$ since any integer can be represented in the above manner, where only a finite number of the components of the sum are non-zero. This also means that $\mathbb{Z}$ can be seen as a dense subring of $\hat{\mathbb{Z}}$.
\end{definition}

\begin{theorem}
	\label{product of all p-adic integers}
	The profinite integers are a profinite topological ring. Furthermore, there exists an isomorphism \[
		\hat{\mathbb{Z}} \cong \prod_{q \ prime} \mathbb{Z}_q
	\]
	of topological rings, where $\mathbb{Z}_q$ is the ring of q-adic integers.
\end{theorem}

\begin{lemma}
	\label{proper closed subgroups of p-adic numbers}
	Any proper nontrivial closed subgroup of the p-adic numbers $\mathbb{Q}_p$ is of the form $p^n\mathbb{Z}_p$ for some $n \in \mathbb{Z}$.
\end{lemma}

\begin{corollary}
	\label{compact open subgroups of p-adic vector space}
	For $n \ge 1$, any compact open subgroup of $\mathbb{Q}_p^n$ is of the form $\prod_{i = 1}^n p^{n_i}\mathbb{Z}_p$.
\end{corollary}

\begin{proof}
	Let $K$ be a compact open subgroup of $\mathbb{Q}_p^n$. Since $K$ is open, it contains a set of the form $\prod_{i = 1}^n p^{l_i}\mathbb{Z}_p$, for $l_i \in \mathbb{Z}$. Furthermore, we can suppose that each $l_i$ is minimal, because of compactness of $K$. At the same time, letting $\pi_i : \mathbb{Q}_p^n \to \mathbb{Q}_p$ be the $i$-th coordinate projection for each $i \in \{1, \dots, n\}$, $\pi_i(K)$ is a compact non-trivial subgroup of $\mathbb{Q}_p$ and therefore equal to $p^{k_i}\mathbb{Z}_p$ for some $k_i \in \mathbb{Z}$ by Lemma \ref{proper closed subgroups of p-adic numbers}. Since each $l_i$ is minimal, this actually implies $k_i = l_i$. At the same time, $K \subset \pi_1(K) \times \dots \times \pi_n(K)$, therefore \[
		\prod_{i=1}^n p^{l_i} \mathbb{Z}_p \subset K \subset \prod_{i = 1}^n p^{l_i}\mathbb{Z}_p.
	\]
\end{proof}

\section{Pontryagin Dual}

Given a $LCA$ group $G$, its Pontryagin Dual is the set of continuous homomorphisms $G \to \unitcircle$, under the compact-open topology (a neighbourhood basis at the identity is given below). Its elements are sometimes referred to as characters of $G$.

\begin{lemma}
	\label{neighbourhood basis on dual}
	For $\chi_0 \in \dual{G}$, defining \[
		V_{K, \epsilon}(\chi_0) := \{\chi \in \dual{G}: \ \sup_{g \in K} \abs{\chi(g) - \chi_0(g)} < \epsilon\},
	\]
	for $K \subset G$ compact and $\epsilon > 0$, the sets $V_{K, \epsilon}(1)$ form a neighbourhood basis at the identity of $\dual{G}$ as $K$ and $\epsilon$ vary (reminder that the topology on $\dual{G}$ is generated by the sets of the form $L(K, O) = \{\chi \in \dual{G}: \ \chi(K) \subset O\}$, as $K \subset G$ compact and $O \subset \unitcircle$ open vary).
\end{lemma}

\begin{proof}
	Let $U = L(K, O)$ be a sub-basic neighbourhood of the identity in $\dual{G}$, where $O \subset \unitcircle$ is open and $\emptyset \neq K \subset G$ is compact. Because $1_{\dual{G}} \in L(K, O)$, $1 \in O$ and therefore there exists $\epsilon > 0$ such that $\{z \in \unitcircle: \ \abs{z - 1} < \epsilon\} \subset O$, by $O$ being open. As such, $V_{K, \epsilon}(1) \subset L(K, O)$. \\
	Because $V_{K, \epsilon}(1)$ is also an open neighbourhood of the identity in $\dual{G}$ and is contained in $L(K, O)$, we can conclude that the identity in $\dual{G}$ admits a neighbourhood basis consisting of finite intersections of the sets of the form $V_{K, \epsilon}(1)$. \\
	Let $K_1, \dots, \cup K_n \subset G$ be compact and $\epsilon_1, \dots, \epsilon_n > 0$. \\
	Then if $K = K_1 \cup \dots \cup K_n$ and $\epsilon = \min\{\epsilon_1, \dots, \epsilon_n\}$, \[
		V_{K, \epsilon}(1) \subset V_{K_1, \epsilon_1}(1) \cap \dots \cap V_{K_n, \epsilon_n}(1).
	\]
	Therefore, sets of this form form a neighbourhood basis at $1_{\dual{G}}$ instead of just a subbasis.
	We can take $K$ to vary over compact neighbourhoods of the identity in $G$ because any compact subset of $G$ is contained in a neighbourhood of such type.
\end{proof}

\begin{theorem}{Pontryagin Duality}
	\label{Duality thm}
	\\ The map \begin{align*}
		\eta_G\colon \  & G \to \ \dual{\dual{G}} \\
		                & g \ \mapsto \ ev_G(g)
	\end{align*}
	is an isomorphism of topological groups for any LCA group G. \\
\end{theorem}

\begin{corollary}
	\label{immediate consequences}
	\item $G \in LCA$ admits a neighbourhood basis at the identity consisting of sets of the form \[
		\{g \in G: \ \abs{\chi(g) - 1} < \epsilon \ \text{for all} \ \chi \in C\},
	\]
	as $\epsilon > 0$ and $C \subset \dual{G}$ compact vary. This is obtained by combining Lemma \ref{neighbourhood basis on dual} and Theorem \ref{Duality thm}.
\end{corollary}

\subsection{Functoriality for LCA groups}

\begin{definition}
	Let G, G' be locally compact abelian groups and $f\colon G' \to G$ a continuous homomorphism. Its dual morphism is defined as \begin{align*}
		\dual{f}\colon\  & \dual{G} \to \ \dual{G'}   \\
		                 & \chi \mapsto \chi \circ  f
	\end{align*}
	It is well defined since f is a continuous homomorphism.
\end{definition}

\begin{prop}
	\label{dual morphism}
	\hfill \begin{enumerate}
		\item $\dual{f}$ is a continuous homomorphism.
		\item $\ker(\dual{f}) = A(Im f)$ and $\orth{(Im \widehat{f})} = ker f$.
		\item If $f\colon G' \to G$ and $g\colon G \to G''$ are both continuous homomorphisms, then $\widehat{g\circ f} = \widehat{f}\circ \widehat{g}$.
		\item The dual of the identity map in G is the identity on $\dual{G}$.
	\end{enumerate}
\end{prop}

\begin{proof}
	It is easy to see that $\dual{f}$ is a homomorphism. Therefore, we only need to check for continuity at the trivial character.
	\begin{enumerate}
		\item[1.] Let $\seq{\chi}{\alpha}{\Lambda}$ be a generalized sequence of characters of G that converges uniformly in compact sets to the trivial character. If $K \subset G'$ is compact, then \[
			      \sup_{k \in K} \ \abs{\chi_\alpha\circ f(h) - 1} \le \sup_{g \in f(K)} \ \abs{\chi_\alpha(g) - 1}.
		      \]
		      Since $f(K)$ is a compact subset of G, the right-hand side converges to 0. Therefore, $\seq{\dual{f}(\chi_\alpha)}{\alpha}{\Lambda}$ converges uniformly on compact sets to the trivial character and $\dual{f}$ is continuous. \\
		\item[2.] Note that
		      \[\dual{f}(\chi) = 1 \ \iff \chi\circ f = 1 \ \iff \chi|_{Im f} = 1,\]
		      therefore $Ker(\dual{f}) = A(Im f)$. \\
		      Analogously, \[
			      g \in \orth{(Im \widehat{f})} \ \iff \chi (f(g)) = 1 \ \text{for all} \ \chi \in \dual{G} \ \iff f(g) = 0.
		      \]
		\item[3.] For $\psi \in \dual{G''}$, $\widehat{g\circ f}(\psi) = \psi\circ (g \circ f) = (\psi \circ g) \circ f = \widehat{f}(\psi \circ g) = \widehat{f}(\widehat{g}(\psi))$.
	\end{enumerate}
\end{proof}

\begin{corollary}
	\label{relations of dual morphism}
	\hfill \begin{enumerate}
		\item If $f$ is invertible with inverse $\inv{f}$, then $\dual{f}$ is also invertible with inverse $\dual{\inv{f}}$. In particular, if $G, G' \in LCA$ such that $\dual{G} \cong \dual{G'}$, then $G \cong G'$.
		\item If f has dense image, then $\widehat{f}$ is injective.
		\item If f is injective, then $\widehat{f}$ has dense image.
	\end{enumerate}
\end{corollary}

\section{Restricted products}

Arbitrary products of locally compact groups need not remain locally compact, so restricted products are introduced to deal with this issue.

\begin{definition}
	\label{Restricted product definition}
	Let $(G_i)_{i \in I}$ be a family of locally compact groups such that the set $J \subset I$ of indices for which $G_i$ contains a compact open subgroup has finite complement. For $i \in J$, we choose a compact open subgroup $H_i$, and for $i \in I\setminus J$, we take $H_i = G_i$. \\
	The restricted product $\prod_{i \in I}^{'} (G_i, H_i)$ is defined as \[
		\{(g_i)_{i \in I} \in \prod_{i \in I}G_i: \ g_i \in H_i \ \text{for all but a finite number of indices}\}.
	\]
	It is easy to see that this is a subgroup of the usual cartesian product.
	The topology in the restricted product is generated by a neighbourhood basis at the identity given by sets of the form\[
		\prod_{i \in I} U_i,
	\]
	where each $U_i$ is an identity neighbourhood in $G_i$ and  $U_i = H_i$ for all but a finite number of indices. \\
\end{definition}

\begin{prop}
	\label{properties of restricted product}
	Given $(G_i, H_i)_{i \in I}$ as above, let $G$ be the aforementioned restricted product. \begin{itemize}
		\item $G$ is a locally compact topological group.
		\item A subset $Y \subset G$ has compact closure if and only if $Y \subset \prod_{i \in I}K_i$, where each $K_i$ is compact in $G_i$ and $K_i = H_i$ for all but finitely many indices.
		\item The canonical maps $\iota_i\colon G_i \to \prod_{i \in I}^{'}(G_i, H_i)$ are isomorphisms under their images.
		\item If the $G_i$'s are locally compact abelian, then \[
			      \dual{G} \cong \prod_{i \in I}^{'}(\dual{G_i}, K(H_i)),
		      \]
		      where $K(H_i) = \begin{cases*}
				      A(H_i), \ \text{if} \ H_i \ \text{is compact and open} \\
				      \dual{G_i} \ \text{otherwise}.
			      \end{cases*}$
	\end{itemize}
\end{prop}

\noindent Proofs can be found in Proposition 5-1 and Theorem 5-4 of \cite{RamakrishnanValenza1999}. \\

It is worth noting that if each $G_i$ also admits a topological ring structure, then the restricted product is also a topological ring, by an analogous argument to the one that proves the restricted product is a topological group.

\bibliographystyle{unsrt}
\bibliography{references}

\end{document}